\documentclass[10pt]{amsart}

\theoremstyle{definition}

\newtheorem{prop}{Proposition}
\newtheorem{lemma}{Lemma}
\newtheorem{cor}{Corollary}
\newtheorem{theorem}{Theorem}

\newcommand{\cal}{\mathcal}

\title[Normality of the twistor space of a $5$-manifold with a $SO(3)$-structure \hfill]{Normality of the twistor space of a $5$-manifold with an irreducible $SO(3)$-structure}
\author{Johann Davidov}
\address{Institute of Mathematics and Informatics \\
Bulgarian Academy of Sciences\\ Acad. G.Bonchev st. Bl.8 \\
1113 Sofia\\ Bulgaria\\ \newline \centerline{and} \newline "L.Karavelov" Civil Engineering Higher
School, 175 Suhodolska st. 1373 Sofia, Bulgaria }

\email{jtd@math.bas.bg}

\begin{document}

\maketitle


\begin{abstract}

A manifold with an irreducible $SO(3)$-structure is a $5$-manifold $M$ whose structure group can be
reduced to the group $SO(3)$, non-standardly imbedded in $SO(5)$. The study of such manifolds has
been initiated by M. Bobie\'nski and P. Nurowski who, in particular,  have shown that one can define four $CR$-structures on a
twistor-like $7$-dimensional space associated to $M$.  In the
present paper it is observed that these $CR$-structures are induced by almost
contact metric structures. The purpose of the paper is to study the problem of normality of these
structures. The main result gives necessary and sufficient condition for normality  in geometric terms of the base manifold $M$.
Examples illustrating this result are presented at the end of the paper.

\vspace{0,1cm} \noindent 2000 {\it Mathematics Subject Classification}. 53C28; 53D15, 53B15.

\vspace{0,1cm} \noindent {\it Key words: irreducible $SO(3)$-structures, twistor spaces, almost
contact metric structures.}

\end{abstract}


\thispagestyle{empty}

\section{Introduction}

It is well-known that the group $SO(3)$ has a unique irreducible representation in dimension five
which gives a non-standard embedding of $SO(3)$ in the group $SO(5)$. Recently M. Bobie\'nski
and P. Nurowski \cite{BN} have studied five-dimensional oriented Riemannian manifolds admitting an
irreducible $SO(3)$-structure, meaning that their tangent frame bundle has a $SO(3)$-subbundle,
$SO(3)$ being non-standardly embedded in $SO(5)$. In general, the Levi-Civita connection does not
preserve this subbundle, which provides an example of the so-called non-integrable geometric
structures. A framework for studying such structures has been outlined by T. Friedrich in
\cite{Fr}. Among the irreducible $SO(3)$-structures the so-called nearly integrable structures are of special interest.
An important feature of these structures is that they admit a (unique) characteristic connection,
i.e. a metric $SO(3)$-connection with totally skew-symmetric torsion \cite{BN}. Interesting geometric properties
of such structures on concrete homogeneous manifolds have been
given by  I. Agricola, J. Becker-Bender, T. Friedrich \cite{ABF}, and a class of nearly integrable
irreducible $SO(3)$-structures on five-dimensional Lie groups has been studied by S. Chiossi, A.
Fino \cite{CF}. We also note that the topological obstructions for existence of an irreducible
$SO(3)$-structure have been found in \cite{B} and \cite{ABF}.

  It has been observed in \cite{BN} that every manifold $M$ with an irreducible $SO(3)$-structure
admits a twistor-like space. This is a $2$-sphere bundle ${\Bbb T}$ over $M$ on which one can
define four almost $CR$-structures in a way that resembles the twistor construction. In fact, these
almost CR-structures are induced by almost contact metric structures and the main purpose of this
paper is to find the geometric conditions on the base manifold $M$ under which these almost contact
structures are normal. Recall that normality is an important property of an almost contact manifold
$N$, which means that the product manifold $N\times S^1$ is a complex manifold with the complex
structure induced by the almost contact one (cf., for example, \cite{Blair}).

  In Section 2 of the present paper we collect basic facts about manifolds with an irreducible
$SO(3)$-structure. The twistor space ${\Bbb T}$ of such a manifold and the four almost contact metric structures on it are defined
in Section 3. The main result of the paper is proved in Section 4. It is not hard to see that three of the almost contact metric structures
on ${\Bbb T}$ are not normal. In order to state the normality result for
the fourth one, we introduce a specific
tensor $Q$ on the base manifold by means of the curvature of its characteristic connection. We show
that this almost contact metric structure on the twistor space ${\Bbb T}$ is normal if and only if
$Q$ and certain components of the torsion and the curvature of the characteristic connection
vanish. Two examples illustrating this result are given at the end of the paper. One of them
provides a new example of a non-normal almost contact metric structure which induces an integrable almost $CR$-structure
(as is well-known, any normal structure induces an integrable structure \cite{I}). Some of the computations in Chapter 4 are used
in an Addendum for giving a new proof of the integrability result in \cite{BN} for the almost $CR$-structures on ${\Bbb T}$.
This proof reveals the role of the geometric conditions on the base manifold obtained their for
integrability/non-integrability of these structures.

\smallskip

\centerline{\small ACKNOWLEDGMENT}

\smallskip

I would like to thank D. Blair and Y. Matsushita for their interest to this paper and helpful discussions.

\section{Irreducible $SO(3)$-structures}

\par

\subsection{The irreducible representation of $SO(3)$ on ${\Bbb R}^5$}$\\$

It is well-known that the irreducible finite-dimensional representations of $SO(3)$ are
odd-dimensional and there is a unique irreducible representation of $SO(3)$ on the space ${\Bbb
R}^{2l+1}$, $l=0,1,...$.  Following \cite{BN}, we shall describe the unique irreducible
(orthogonal) representation of $SO(3)$ on ${\Bbb R}^5$ as well as an $SO(3)$-invariant symmetric
$3$-form on ${\Bbb R}^5$ which plays a crucial role for defining $5$-manifolds with
$SO(3)$-structure.

 \par

 Let $e_1,..., e_5$ be the standard basis of ${\Bbb R}^5$. Denote by $\mu$ the
isomorphism of ${\Bbb R}^5$ onto the space of symmetric traceless real $3\times
3$-matrices given by
$$
x=\sum_{i=1}^5 x_ie_i\to \mu(x)=\left[
\begin{array}{ccc}
\displaystyle{\frac{x_1}{\sqrt 3}}-x_4 & x_2 & x_3 \\
x_2 &\displaystyle{ \frac{x_1}{\sqrt 3}}+x_4& x_5 \\
x_3 & x_5 & -2\displaystyle{\frac{x_1}{\sqrt 3}}
\end{array} \right].
$$
Then one can define an irreducible faithful representation $\rho$ of $SO(3)$ on
${\Bbb R}^5$  setting
$$
\rho(h)\cdot x=\mu^{-1}(h\mu(x)h^{-1}),\quad h\in SO(3).
$$

In this way we obtain a non-standard
smooth embedding
$$
\imath: SO(3)\hookrightarrow SO(5)
$$
such that $\imath(SO(3))$ acts on ${\Bbb R}^5$ irreducibly. Further
on, we shall often consider $SO(3)$ as a subgroup of $SO(5)$ by
means of the embedding $\imath$.

\smallskip

  The characteristic polynomial $P_x(\lambda)=det (\mu(x)-\lambda I)$ of the matrix
$\mu(x)$ has the form
$$
P_x(\lambda)=-\lambda^3+g(x,x)\lambda+\frac{2\sqrt 3}{9}\Upsilon(x,x,x),
$$
where
$$
g(x,x)=x_1^2+...+x_5^2
$$
is the standard metric of ${\Bbb R}^5$ and
$\Upsilon(x,x,x)=\displaystyle{\frac{3\sqrt 3}{2}}det\,\mu(x)$ is given by
\begin{equation}\label{gamma}
\begin{array}{lll}
\Upsilon(x,x,x)&=&\displaystyle{\frac{1}{2}}x_1(6x_2^2+6x_4^2-2x_1^2-3x_3^2-3x_5^2)\\[6pt]
& &+\displaystyle{\frac{3\sqrt 3}{2}}x_4(x_5^2-x_3^2)+3\sqrt 3 x_2x_3x_5.
\end{array}
\end{equation}
Obviously $P_{\rho(h)\cdot x}=P_x$, hence the polynomial
$\Upsilon$ is $\rho$-invariant.
Denote the $SO(3)$-invariant symmetric $3$-form
on ${\Bbb R}^5$ corresponding to the $SO(3)$-invariant polynomial
$\Upsilon$   also by $\Upsilon$.

An oriented orthonormal basis $a_1,...,a_5$ of ${\Bbb R}^5$ will be
called {\it adapted} if $\Upsilon(x,x,x)$, $x\in {\Bbb R}^5$, is
given by the right-hand side of (\ref{gamma}) with $x_1,...,x_5$
being the coordinates of $x$ with respect to the basis
$a_1,...,a_5$. Obviously, the standard basis $e=(e_1,...,e_5)$ of
${\Bbb R}^5$ is adapted. In view of the $SO(3)$-invariance of
$\Upsilon$, it is clear that the action of $SO(3)$ on ${\Bbb R}^5$
preserves the set of the adapted bases. Moreover, this action is
transitive since, by \cite[Proposition 2.5]{BN}, the stabilizer of
$\Upsilon$ under the standard action of $O(5)$ on the symmetric
$3$-forms on ${\Bbb R}^5$ coincides with $SO(3)$.

\par

  Now consider each $W\in \otimes^2{\Bbb R}^5$ as
an endomorphism of ${\Bbb R}^5$ given by $g(W(x),y)=g(W,x\otimes y)$, $x,y\in
{\Bbb R}^5$, and set
$$
\widehat\Upsilon(W)(x)=4\sum_{j=1}^5 \Upsilon_{W(e_j)}\circ \Upsilon_{e_j}(x).
$$
Then $\widehat\Upsilon$ preserves the decomposition $\otimes^2{\Bbb R}^5=\Lambda^2{\Bbb R}^5\oplus
\odot^2{\Bbb R}^5$ and according to \cite[Proposition 3.3]{BN} we have the orthogonal decomposition
\begin{equation}\label{ort}
\otimes^2{\Bbb R}^5
=\Lambda^2_3\oplus\Lambda^2_7\oplus\odot^2_1\oplus\odot^2_5\oplus\odot^2_9,
\end{equation}
where
$$
\begin{array}{lll}
\Lambda^2_3=\{F\in\otimes^2{\Bbb R}^5: \widehat\Upsilon(F)=7F\},\quad
\Lambda^2_7=\{F\in\otimes^2{\Bbb R}^5: \widehat\Upsilon(F)=-8F\},\\[6pt]
\odot^2_1=\{S\in\otimes^2{\Bbb R}^5: \widehat\Upsilon(S)=14S\}
=\{S=\lambda g: \lambda\in{\Bbb R}\},\\[6pt]
\odot^2_5=\{S\in\otimes^2{\Bbb R}^5: \widehat\Upsilon(S)=-3S\},\quad
\odot^2_9=\{S\in\otimes^2{\Bbb R}^5: \widehat\Upsilon(S)=4S\}.\\
\end{array}
$$
The representations $\Lambda^2_j$ and $\odot^2_k$ of $SO(3)$ are
irreducible; the indices $j$ and $k$ indicate their dimensions.

\par

The space $\Lambda^2_3$ will be used in the next section to define a twistor
space.

If $a=\{a_1,...,a_5\}$ is an adapted basis of ${\Bbb R}^5$, set
\begin{equation}\label{kappa}
\begin{array}{lll}
\kappa_1=\kappa_1(a)=\sqrt 3 a_1\wedge a_5+a_2\wedge a_3+a_4\wedge a_5,\\[6pt]
\kappa_2=\kappa_2(a)=\sqrt 3 a_1\wedge a_3+a_2\wedge a_5+a_3\wedge a_4,\\[6pt]
\kappa_3=\kappa_3(a)=2a_2\wedge a_4+a_3\wedge a_5.
\end{array}
\end{equation}
Then $\{\kappa_1,\kappa_2,\kappa_3\}$ is an orthogonal basis of $\Lambda^2_3$ with
$|\kappa_1|^2=|\kappa_2|^2=|\kappa_3|^2=5$, the metric on $\Lambda^2{\Bbb R}^5$ being defined by
$g(x\wedge y, u\wedge v)=g(x,u)g(y,v)-g(x,v)g(y,u)$.

\subsection{Manifolds with an irreducible  $SO(3)$-structure}$\\$

Let $M$ be a $5$-dimensional manifold. Suppose that the structure group of the tangent bundle
$\pi:TM\to M$ can be reduced to the group $\imath(SO(3))$. Using an atlas of trivializations of
$TM$ whose transition functions take values in $\imath(SO(3))$, we can define a Riemannian metric
$g$ and an orientation on $M$, and can transfer the tensor $\Upsilon$ on ${\Bbb R}^5$ to a rank $3$
tensor on $M$ (since $\Upsilon$ is invariant under $\imath(SO(3))$). We denote this tensor on $M$
again by $\Upsilon$. Taking an adapted basis of ${\Bbb R}^5$ we can get a local oriented
orthonormal frame $E_1,...,E_5$ of $TM$ such that
$$
\begin{array}{c}
\Upsilon(X,X,X)=\displaystyle{\frac{1}{2}}x_1(6x_2^2+6x_4^2-2x_1^2-3x_3^2-3x_5^2)\\[6pt]
+\displaystyle{\frac{3\sqrt 3}{2}}x_4(x_5^2-x_3^2)+3\sqrt 3 x_2x_3x_5
\end{array}
$$
for $X=\sum_{i=1}^4x_iE_i$. A frame with this property will be called {\it adapted}. The set of adapted
bases of tangent spaces of $M$ constitute the total space of a principal $\imath(SO(3))$-bundle
${\cal A}{\cal B}(M)$ that is a reduction of the bundle of oriented orthonormal frames of $(M,g)$.

For every $v\in TM$, the linear map $\Upsilon_v:
T_{\pi(v)}M\to T_{\pi(v)}M$ defined by
$$
g(\Upsilon_v(x),y)=\Upsilon(v,x,y),\quad x,y\in T_{\pi(v)}M,
$$
has the following properties:
\begin{enumerate}
\item[$(i)$]~ $\Upsilon_v$ is totally symmetric, i.e.\\
$ g(u,\Upsilon_v(w))=g(w,\Upsilon_v(u))=g(u,\Upsilon_w(v)), \quad
u,v,w\in T_{\pi(v)}M; $
\item[$(ii)$]~$\Upsilon_v$ is trace-free, $Trace\,\Upsilon_v=0$;
\item[$(iii)$]$\Upsilon_v^2(v)=g(v,v)v$.
\end{enumerate}

One of the important observations in \cite{BN} is that any rank $3$ tensor  on an oriented
Riemannian $5$-manifold with the properties $(i)$ - $(iii)$  determines a reduction of the
structure group of $TM$ to the group $\imath(SO(3))$. Two reductions determined by tensors
$\Upsilon$ and $\widetilde\Upsilon$ are equivalent iff there is an isometry $\varphi$ of $M$ such
that $\varphi^{\ast}(\widetilde\Upsilon)=\Upsilon$.

  We shall say a $5$-manifold $M$ has an
{\it irreducible $SO(3)$-structure}  if the structure group of $TM$ can be reduced to the group
$\imath(SO(3))$. The topological obstructions for existence of such a structure are discussed in
\cite{ABF,B}. Examples of manifolds with irreducible $SO(3)$-structures can be found in
\cite{ABF,BN,B,CF}.

Let $(M,g,\Upsilon)$ be an oriented Riemannian $5$-manifold with an
irreducible $SO(3)$-structure determined by a rank $3$ tensor
$\Upsilon$ having the properties $(i) - (iii)$.

The orthogonal decomposition (\ref{ort}) of $\otimes^2{\Bbb R}^5$
gives rise to an orthogonal decomposition
\begin{equation}\label{ort'}
\otimes^2TM
=\Lambda^2_3TM\oplus\Lambda^2_7TM\oplus\odot^2_1TM\oplus\odot^2_5TM\oplus\odot^2_9TM,
\end{equation}
where $\Lambda^2_3TM$ is the associated bundle
${\cal A}{\cal B}(M)\times_{\imath(SO(3))}\Lambda^2_3$, etc.

Every $SO(3)$-connection $\nabla$ on $M$ is metric ($\nabla g=0$)
and preserves $\Upsilon$ ($\nabla\Upsilon=0$) as well as  the
decomposition (\ref{ort'}).

The curvature operator of $\nabla$ at any point $p\in M$ takes its values in $\Lambda^2_3T_pM$
($\cong so(3))$. Note that for the curvature tensor $R$ and the curvature operator ${\cal
R}:\Lambda^2TM\to \Lambda^2 TM$ we adopt the following definitions:
$$R(X,Y,Z)=-\nabla_X\nabla_YZ+\nabla_X\nabla_YZ+\nabla_{[X,Y]}Z,$$
$$g({\cal R}(X\wedge Y), Z\wedge U)=g(R(X,Y)Z,U),$$
for $X,Y,Z,U\in TM$.

\section{The twistor space of a manifold with an irreducible $SO(3)$-structure and almost contact metric structures on it}

As in \cite{BN}, define the twistor space of $M$ as the sphere bundle
$$
{\Bbb T}=\{\sigma\in\Lambda^2_3TM: |\sigma|^2=5\}.
$$ The radius of the fibre is chosen so that
$\kappa_1,\kappa_2,\kappa_3$ defined via (\ref{kappa}) by means of an adapted frame be sections of
${\Bbb T}$. We can define four $CR$-structures on the manifold ${\Bbb T}$ in the following way (\cite{BN}).

Henceforward the restrictions to $\Lambda^2_3TM$ and ${\Bbb
T}$ of the bundle projection $\pi:\Lambda^2TM\to M$ will be denoted again by $\pi$. The tangent
bundle of ${\Bbb T}$ will be considered as a subbundle of the tangent bundle of the manifold $\Lambda^2_3TM$. For
every $\sigma\in{\Bbb T}$, the vertical space ${\cal V}_{\sigma}$ at $\sigma$ of the bundle
$\pi:{\Bbb T}\to M$ (i.e. the tangent space at $\sigma$ of the fibre through $\sigma$) will be
denoted by ${\cal V}_{\sigma}$; clearly ${\cal V}_{\sigma}$ is the orthogonal complement of
$\sigma$ in $\Lambda^2_3T_{\pi(\sigma)}M$.

Let $\nabla$ be a $SO(3)$-connection on $M$. Then, as we have mentioned, $\nabla$ preserves
$\Lambda^2_3TM$, so it induces a connection on this vector bundle which we shall denote also by
$\nabla$. For every $\sigma\in{\Bbb T}$, the horizontal subspace ${\cal H}_{\sigma}$ of
$T_{\sigma}(\Lambda^2_3TM)$ with respect to $\nabla$ is tangent to ${\Bbb T}$ since $\nabla$ is
metric. Thus, the tangent bundle of the twistor space is the direct sum of its vertical bundle ${\cal V}$
and the horizontal bundle ${\cal H}$ of $\nabla$: $T{\Bbb T}={\cal V}\oplus{\cal H}$.

For $\sigma\in {\Bbb T}$, set
$$
\xi_{\sigma}=\frac{1}{4}\ast(\sigma\wedge\sigma),
$$
where $\ast:\Lambda^4TM\to TM$ is the Hodge star operator. Denote by $H^{\sigma}$ the orthogonal
complement of $\xi_{\sigma}$ in $T_{\pi(\sigma)}M$ and give $H^{\sigma}$ the orientation that
yields the orientation of $T_{\pi(\sigma)}M$ via the decomposition $T_{\pi(\sigma)}M={\Bbb
R}\xi_{\sigma}\oplus H^{\sigma}$.  Since $(H^{\sigma},g)$ is an oriented
Euclidean four-dimensional vector space, $\Lambda^2H^{\sigma}$ decomposes into self-dual and
anti-self-dual parts, $\Lambda^2H^{\sigma}=\Lambda^2_{+}H^{\sigma}\oplus\Lambda^2_{-}H^{\sigma}$.
Denote by $\sigma_{\pm}$ the components of $\sigma$ with respect to this decomposition.
The $2$-vectors $\sigma_{\pm}$  determine two  complex structures $J^{\sigma}_{\pm}$ on the vector
space $H^{\sigma}$ compatible with its metric and $\pm$-orientation. These are given by
\begin{equation}\label{J}
g(J^{\sigma}_{\pm}X,Y)=\frac{2}{2\pm 1}g(\sigma_{\pm}, X\wedge Y),~ X,Y\in H^{\sigma}.
\end{equation}

For $\sigma\in{\Bbb T}$, denote by $(H^{\sigma})^h_{\sigma}$ the horizontal lift at $\sigma$  of
the vector space $H^{\sigma}$ with respect to $\nabla$. Let ${\cal D}\to {\Bbb T}$ be the subbundle
of $T{\Bbb T}$ whose fibre at $\sigma\in{\Bbb T}$ is $${\cal D}_{\sigma}={\cal V}_{\sigma}\oplus
(H^{\sigma})^h_{\sigma}.$$
The fibre of the twistor bundle $\pi:{\Bbb T}\to M$ is the unit sphere and we denote its standard complex structure
by ${\cal J}$.  Then, as is usual in the twistor theory,  we define  almost complex structure $\cal
J^{(n)}_{\pm}$, $n=1,2$, on the bundle ${\cal D}$ setting
$$
\cal J^{(n)}_{\pm}V=(-1)^{n+1}{\cal J}~\mbox{ for }~ V\in{\cal V}_{\sigma},\quad
\cal J^{(n)}_{\pm}X^h_{\sigma}=(J^{\sigma}_{\pm}X)^h_{\sigma}~\mbox{ for }~ X\in H^{\sigma}.
$$
In this way we obtain four almost $CR$-structures $({\cal D},\cal
J^{(n)}_{\pm})$ on the manifold ${\Bbb T}$.

We extend $J^{\sigma}_{\pm}$ to partially complex structures
$\varphi^{\sigma}_{\pm}$ on the space $T_{\pi(\sigma)}M$,  setting
$\varphi^{\sigma}_{\pm}(\xi_{\sigma})=0$ and $\varphi^{\sigma}_{\pm}=J^{\sigma}_{\pm}$ on
$H^{\sigma}$. Then $\varphi^{\sigma}_{\pm}$ is $g$-skew-symmetric,
$(\varphi^{\sigma}_{\pm})^2X=-X+g(X,\xi_{\sigma})\xi_{\sigma}$ for every $X\in T_{\pi(\sigma)}M$,
and $Im\,\varphi^{\sigma}_{\pm}=H^{\sigma}$, $Ker\,\varphi^{\sigma}_{\pm}={\Bbb R}\xi_{\sigma}$.
Now we define partially complex structures
$\Phi_{\pm}^{(n)}$, $n=1,2$, of rank $6$ on the manifold ${\Bbb T}$ setting $\Phi_{\pm}^{(n)}=\cal
J^{(n)}_{\pm}$ on ${\cal V}$  and
$\Phi_{\pm}^{(n)}X^h_{\sigma}=(\varphi^{\sigma}_{\pm}X)^h_{\sigma}$ for $X\in T_{\pi(\sigma)}M$.

\smallskip

\noindent {\it Remark}. By "partially complex structure" of rank $2k$ on a manifold $N$, $0<2k\leq
dim\,N$, we mean an endomorphism ${\cal F}$ of $TN$ of rank $2k$ such that ${\cal F}^3+{\cal F}=0$.
Clearly, given such a structure, we have $TN=Im {\cal F}\oplus Ker {\cal F}$ and ${\cal F}$ is an
almost complex structure on the vector bundle $Im {\cal F}$, so the name. Partially complex
structures are also called $f$-structures.

\smallskip

For any fixed $t>0$, define a metric $h_t$ on ${\Bbb T}$ by
$$
h(X^h_{\sigma}+V, Y^h_{\sigma}+W)=g(X,Y)+tg(V,W),
$$
where $V,W\in{\cal V}_{\sigma}$, $X,Y\in T_{\pi(\sigma)}M$ (this is the so-called canonical
variation of the Riemannian metric $g$ \cite{Besse}).
Set $\chi_{\sigma}=(\xi_{\sigma})^h_{\sigma}$. Then $(\Phi_{\pm}^{(n)},\chi,h_t)$ are almost contact metric structures inducing the $CR$-structures $({\cal D},\cal
J^{(n)}_{\pm})$. We
refer to \cite{Blair} for general facts about (almost)  contact metric structures.

\section{Normality of the almost contact metric structures on the twistor space}

Recall that any almost contact structure $(\varphi,\xi,\eta)$ on a
manifold $N$ induces an almost complex structure on the manifold
$N\times S^1$ in a natural way (\cite{Blair}). An almost contact
metric structure on $N$ is said to be normal if the induced almost
complex structure on $N\times S^1$ is integrable. It is well-known
that $(\varphi,\xi,\eta)$ is a normal structure if and only if the
tensor $N^{(1)}(X,Y)=\varphi ^{2}[X,Y]+ [\varphi X,\varphi
Y]-\varphi [\varphi X,Y] -\varphi [X,\varphi Y] +d\eta(X,Y)\xi$
vanishes (see, for example, \cite{Blair}; the definition of $d\eta$
used here is twice the one in \cite{Blair}). This tensor for the
almost contact structure $(\Phi^{(n)}_{\pm},\chi,h_t)$ will be
denote by $N^{(n)}_{\pm}$.

\subsection {Certain technicalities}$\\$

For any adapted basis  $a=\{a_1,...,a_5\}$  of ${\Bbb R}^5$, define an  orthogonal basis of $\Lambda^2_3$
$\{\kappa_1(a),\kappa_2(a),\kappa_3(a)\}$ by means of (\ref{kappa}).

\smallskip

The proof of the next statement contains  formulas that will
also be used latter on.

\begin{lemma}\label{orient}
All bases $\{\kappa_1(a),\kappa_2(a),\kappa_3(a)\}$ determine the same orientation on the vector space $\Lambda^2_3$.
\end{lemma}

\begin{proof}
Let $a=\{a_1,...,a_5\}$  and $a'=\{a_1',...,a_5'\}$ be two adapted
bases of ${\Bbb R}^5$. As we have remarked, the group $SO(3)$ acts
transitively on the set of adapted bases, so there is a matrix $h$
in $SO(3)$ such that $\imath (h)$ sends the basis $a$ to the basis
$a'$. Any such a matrix $h$ can be represented as the product
$h_{\psi}.h_{\theta}.h_{\varphi}$ of the $SO(3)$-matrices
$$
h_{\psi}=\left[\begin{matrix}
\cos\psi & \sin\psi & 0 \\
-\sin\psi & \cos\psi & 0 \\
 0   &  0 & 1
\end{matrix}\right],
h_{\theta}=\left[\begin{matrix}
0 & 0 & 1 \\
\cos\theta & \sin\theta & 0 \\
-\sin\theta & \cos\theta & 0
\end{matrix}\right],
h_{\varphi}=\left[\begin{matrix}
-\sin\varphi & \cos\varphi & 0 \\
0 & 0 & 1 \\
\cos\varphi & \sin\varphi & 0
 \end{matrix}\right].
$$
The matrix $\imath(h_{\varphi})$ transforms the bases
$a=\{a_1,...,a_5\}$ to the bases $b=\{b_1,...,b_5\}$ where
\begin{equation}\label{b}
\begin{array}{cll}
b_1=-\displaystyle{\frac{1}{2}}a_1-\displaystyle{\frac{\sqrt{3}}{2}}\,a_2\sin 2\varphi+\displaystyle{\frac{\sqrt{3}}{2}}\,a_4\cos 2\varphi,\\[8pt]
b_2=-a_3\sin\varphi+a_5\cos\varphi, \quad b_3=a_2\cos 2\varphi+a_4\sin 2\varphi,\\
b_4=-\displaystyle{\frac{\sqrt{3}}{2}}a_1+\displaystyle{\frac{1}{2}}a_2\sin 2\varphi
-\displaystyle{\frac{1}{2}}a_4\cos 2\varphi,\quad b_5=a_3\cos\varphi+a_5\sin\varphi.
\end{array}
\end{equation}
For this basis, we have
\begin{equation}\label{k(b)}
\begin{array}{c}
\kappa_1(b)=-\sin\varphi\,\kappa_1(a)-\cos\varphi\,\kappa_2(a),\quad
\kappa_2(b)=\kappa_3(a), \\[6pt]
\kappa_3(b)=\cos\varphi\,\kappa_1(a)-\sin\varphi\,\kappa_2(a).
\end{array}
\end{equation}
Hence $\{\kappa_1(a),\kappa_2(a),\kappa_3(a)\}$ and
$\{\kappa_1(b),\kappa_2(b),\kappa_3(b)\}$ yield the same
orientation.

The matrix $\imath(h_{\theta})$ transforms $b=\{b_1,...,b_5\}$ to $c=\{c_1,...,c_5\}$ with
\begin{equation}\label{c}
\begin{array}{c}
c_1=-\displaystyle{\frac{1}{2}}b_1+\displaystyle{\frac{\sqrt{3}}{2}}\,b_2\sin
2\theta-\displaystyle{\frac{\sqrt{3}}{2}}\,b_4\cos 2\theta,\quad
c_2=b_3\cos\theta+b_5\sin\theta, \\[6pt]
c_3=-b_3\sin\theta+b_5\cos\theta,\>
c_4=\displaystyle{\frac{\sqrt{3}}{2}}b_1+\displaystyle{\frac{1}{2}}b_2\sin 2\theta-\displaystyle{\frac{1}{2}}b_4\cos 2\theta,\\[6pt]
c_5=b_2\cos 2\theta +b_4\sin 2\theta.
\end{array}
\end{equation}
It follows that
\begin{equation}\label{k(c)}
\begin{array}{c}
\kappa_1(c)=\kappa_3(b),\quad
\kappa_2(c)=-\cos\theta\,\kappa_1(b)+\sin\theta\,\kappa_2(a), \\[6pt]
\kappa_3(c)=-\sin\theta\,\kappa_1(b)-\cos\theta\,\kappa_2(b).
\end{array}
\end{equation}
Finally, the matrix $\imath(h_{\psi})$ sends any adapted basis $c=\{c_1,...,c_5\}$ to the adapted
bases $a'=\{a_1',...,a_5'\}$ for which
\begin{equation}\label{a'}
\begin{array}{c}
a_1'=c_1,\quad a_2'= c_2\cos 2\psi+c_4\sin 2\psi,\\[6pt]
a_3'=c_3\cos\psi+c_5\sin\psi\\[6pt]
a_4'=-c_2\sin 2\psi+c_4\cos 2\psi,\quad a_5'=-c_3\sin\psi+c_5\cos\psi.
\end{array}
\end{equation}
This implies that
\begin{equation}\label{k(a')}
\begin{array}{c}
\kappa_1(a')=\cos\psi\,\kappa_1(c)-\sin\psi\,\kappa_2(c),\quad
\kappa_2(a')=\sin\psi\,\kappa_1(c)+\cos\psi\,\kappa_2(a),\\[6pt]
\kappa_3(a')=\kappa_3(c).
\end{array}
\end{equation}
Now, the lemma follows from (\ref{k(b)}),(\ref{k(c)}),
(\ref{k(a')}).
\end{proof}

\medskip

For any $\kappa\in\Lambda^2{\Bbb R}^5$, denote by $S_{\kappa}$ the
skew-symmetric endomorphism of $\Lambda^2{\Bbb R}^5$ corresponding
to $\kappa$:\, $g(S_{\kappa}(x),y)=g(\kappa,x\wedge y)$.

\smallskip
Let $(\kappa_1,\kappa_2,\kappa_3)$ be the basis of $\Lambda^2_3$
defined via (\ref{kappa}) by means of the standard basis of ${\Bbb
R}^5$. Then
$$
[S_{\kappa_1},S_{\kappa_2}]=-S_{\kappa_3},\quad
[S_{\kappa_2},S_{\kappa_3}]=-S_{\kappa_1}, \quad
[S_{\kappa_3},S_{\kappa_1}]=-S_{\kappa_2}.
$$
Thus the map $\Lambda^2_3\ni\kappa\to -S_{\kappa}$ can be used to defined on $\Lambda^2_3$ the
structure of a $3$-dimensional Lie algebra isomorphic to $so(3)$. Considering $\Lambda^2_3$ with
this structure, the negative Killing form coincides with the metric of $\Lambda^2_3$ induced by the
metric of $\Lambda^2{\Bbb R}^5$.

\smallskip

Fix the orientation on $\Lambda^2_3$ defined by Lemma~\ref{orient}.
Then we have the following.

\begin{lemma}\label{ort-basis}
For every oriented orthogonal basis $q_1,q_2,q_3$ of $\Lambda^2_3$ with
$|q_1|^2=|q_2|^2=|q_3|^2=5$, there is an adapted basis $a=(a_1,...,a_5)$ of ${\Bbb R}^5$ such that
$$q_1=\kappa_1(a), \quad q_2=\kappa_2(a),\quad  q_3=\kappa_3(a).$$
\end{lemma}

\begin{proof}
Consider $\Lambda^2_3$ with the Lie algebra structure isomorphic to $so(3)$ we have defined above.
Then it follows from (\ref{k(b)}),(\ref{k(c)}), (\ref{k(a')}) that the standard action of
$\imath(SO(3))$ on $\Lambda^2_3\subset \Lambda^2{\Bbb R}^5$ coincides with the adjoint action of
$SO(3)$: $Ad_{h}(S_{\kappa})=\imath(h)^{-1}\circ S_{\kappa}\circ \imath(h)$ for $h\in SO(3)$,
$\kappa\in \Lambda^2_3$. Let $(\kappa_1,\kappa_2,\kappa_3)$ be the basis of $\Lambda^2_3$ defined
by means of the standard basis $(e_1,...,e_5)$ of ${\Bbb R}^5$. Then $(q_1,q_2,q_3)$ can be sent to
$(\kappa_1,\kappa_2,\kappa_3)$  by an orthogonal transformation of $\Lambda^2_3$. Any such a
transformation has the form $Ad_h$ for some $h\in SO(3)$, so there is  $h\in SO(3)$ for which
$S_{q_j}=\imath(h)^{-1}\circ S_{\kappa_j}\circ \imath(h)$ for $j=1,2,3$. Then the basis
$a=\{a_l=\imath(h)^{-1}(e_l): l=1,...,5\}$ is adapted and computing the values of $S_{q_j}$ at
$a_1,...,a_5$ we see that $S_{q_j}=S_{\kappa_j(a)}$, $j=1,2,3$. Therefore $q_j=\kappa_j(a)$.
\end{proof}

Now  take a point $\sigma\in{\Bbb T}$. By Lemma~\ref{ort-basis}, there is an adapted basis
$(a_1,...,a_5)$ of $T_{\pi(\sigma)}M$ such that $\sigma=2a_2\wedge a_4+a_3\wedge a_5$. In this
basis, $\xi_{\sigma}=\frac{1}{4}\ast(\sigma\wedge\sigma)=-a_1$ and $(a_2,a_4,a_3,a_5)$ is an oriented orthonormal basis of
$H^{\sigma}$. Note that $\imath_{\xi_{\sigma}}\sigma=0$. Considering $\Lambda^2H^{\sigma}$
as a subspace of $\Lambda^2T_{\pi(\sigma)}M$, we have
$\Lambda^2H^{\sigma}=\{\tau\in\Lambda^2T_{\pi(\sigma)}M: \imath_{\xi_{\sigma}}\tau=0\}$; here and
farther on $\imath:\Lambda^kTM\to \Lambda^{k-1}TM$, $1\leq k\leq 5$, stands for the interior
product, $g(\imath_X\tau,X_1\wedge...\wedge X_{k-1})=g(\tau,X\wedge X_1\wedge...\wedge X_{k-1})$
for $\tau\in\Lambda^kTM$ and $X,X_1,...,X_{k-1}\in TM$. If $\sigma_{\pm}$ are the $\Lambda^2_{\pm}H^{\sigma}$-components of $\sigma$,
we have $\sigma_{+}=\frac{3}{2}(a_2\wedge a_4+a_3\wedge a_5)$, $\sigma_{-}=\frac{1}{2}(a_2\wedge
a_4-a_3\wedge a_5)$ in the adapted basis $(a_1,...,a_5)$ we have chosen above. This shows that
\begin{equation}\label{sigma-pm}
\sigma_{\pm}=\frac{1}{2}(\sigma\pm\imath_{\xi_{\sigma}}(\ast\sigma)).
\end{equation}
Note also that, in view of (\ref{J}),
$$
g(\varphi^{\sigma}_{\pm}X,Y)=\frac{2}{2\pm1}g(\sigma_{\pm}, X\wedge Y),~ X,Y\in T_{\pi(\sigma)}M,
$$
since $\imath_{\xi_{\sigma}}\sigma_{\pm}=0$.

   According to Lemma~\ref{orient}, the vector bundle $\Lambda^2_3TM$ admits an orientation for which
every $(\kappa_1,\kappa_2,\kappa_3)$ defined via (\ref{kappa}) is an oriented frame. Denote by  $\times$
the usual vector product on the oriented Euclidean $3$-dimensional vector space
$(\Lambda^2_3T_pM,g_p)$, $p\in M$. Then the complex structure ${\cal J}$ of the fibre of ${\Bbb T}$ through $\sigma$ is given by
$$
{\cal J}V=-\sigma\times V, \quad V\in{\cal V}_{\sigma}.
$$
Thus
$$
\Phi^{(n)}_{\pm}V=(-1)^n\sigma\times V~\mbox{ for }~ V\in{\cal V}_{\sigma},
$$
$$
\Phi^{(n)}_{\pm}X^h_{\sigma}=(\varphi^{\sigma}_{\pm}X)^h_{\sigma}~\mbox{ for }~ X\in T_{\pi(\sigma)}M.
$$

\subsection {Computation of the tensors $N^{(n)}_{\pm}$}$\\$

Further on, for any $\sigma\in \Lambda^2TM$, the endomorphism $S_{\sigma}$ of $T_{\pi(\sigma)}M$
determined by $\sigma$ and the metric $g$ will also be denoted by $\sigma$.

\smallskip

Let $(x_1,...,x_5)$ be a local coordinate system of $M$ and $(E_1,...,E_5)$ an adapted frame of
$TM$ on $U$. Using this frame, define sections $(\kappa_1,\kappa_2,\kappa_3)$ of $\Lambda^2_3TM$ by
(\ref{kappa}). Set
\begin{equation}\label{coord}
\tilde x_{\alpha}(\tau)=x_{\alpha}\circ\pi(\tau),\,1\leq\alpha\leq
5,~y_i(\tau)=\frac{1}{5}g(\tau,\kappa_i\circ\pi(\tau)),\,1\leq i\leq 3,
\end{equation}
for $\tau\in\Lambda^2_3TM$. Then $(\tilde x_{\alpha},y_{i})$ is a local coordinate system of the
manifold $\Lambda^2_3TM$. For each vector field
$$X=\sum_{\alpha=1}^5 X^{\alpha}\frac{\partial}{\partial x_{\alpha}}$$
on $U$, its horizontal lift $X^h$ is given by
\begin{equation}\label{h-lift}
X^h=\sum_{\alpha=1}^5 (X^{\alpha}\circ\pi)\frac{\partial}{\partial\widetilde x_{\alpha}}
-\sum_{i,j=1}^3y_i(g(\nabla_X\kappa_i,\kappa_j)\circ\pi)\frac{\partial}{\partial y_j}.
\end{equation}
For  $\tau\in\Lambda^2_3TM$, we have
\begin{equation}\label{bra-2h}
[X^h,Y^h]_{\tau}=[X,Y]^h_{\tau}+R(X,Y)\tau,
\end{equation}
where $R(X,Y)\tau$ is the curvature of the connection $\nabla$ on the vector bundle
$\Lambda^2_3TM$.

Set
$$
f^{\pm}_{\alpha \beta}(\sigma)=g(\varphi^{\sigma}_{\pm}E_{\alpha},E_{\beta})
$$
for $\sigma\in(\pi|{\Bbb T})^{-1}(U)$. Then
\begin{equation}\label{Phi}
\begin{array}{c}
\Phi_{\pm}^{(n)} X^h=\sum_{\alpha,\beta=1}^ 5 (g(X,E_{\alpha})\circ\pi)f^{\pm}_{\alpha \beta}E_{\beta}^h,\\[6pt]
(\Phi_{\pm}^{(n)})^{2} X^h=\sum_{\alpha,\beta,\gamma =1}^ 5 (g(X,E_{\alpha})\circ\pi)f^{\pm}_{\alpha \beta}f^{\pm}_{\beta \gamma}E_{\gamma}^h.
\end{array}
\end{equation}

It is easy to see that, in the local coordinates  $\tilde
x_{\alpha}$, $y_i$ on ${\Bbb T}$,
\begin{equation}\label{xi}
\begin{array}{c}
\xi_{\sigma}=\displaystyle{(\frac{1}{2}y_1^2+\frac{1}{2}y_2^2-y_3^2)E_1-\sqrt 3 y_1y_2E_2 + \sqrt 3
y_1y_3E_3} -\displaystyle{\frac{\sqrt 3}{2}(y_1^2-y_2^2)E_4 - \sqrt 3 y_2y_3E_5}
\end{array}
\end{equation}
for $\sigma=y_1\kappa_1+y_2\kappa_2+y_3\kappa_3$. Then, setting
$$h_{\alpha}(y)=g(\xi_{\sigma},E_{\alpha}),$$ we have
\begin{equation}\label{f}
f^{\pm}_{\alpha \beta}=\frac{1}{2\pm1}\sum_{i=1}^3 y_i[g(\kappa_i,E_{\alpha}\wedge
E_{\beta})\circ\pi \pm \sum_{\epsilon=1}^5 h_{\epsilon}(y)g(\ast\kappa_i,E_{\alpha}\wedge
E_{\beta}\wedge E_{\epsilon})\circ\pi].
\end{equation}

\noindent {\bf Notation}. Fix a point $\sigma\in {\Bbb T}$ and take an adapted frame
$(E_1,...,E_5)$ in a coordinate neighbourhood $U$ of the point $p=\pi(\sigma)$ such that
$$
\nabla E_{\alpha}|_p=0, \, \alpha=1,...,5, ~ \mbox{ and }~ \sigma=(\kappa_3)_p,
$$
where $\kappa_3$ is defined  via (\ref{kappa}) by means of the frame $(E_1,...,E_5)$. We also use
this frame to define $\kappa_1$ and $\kappa_2$ by (\ref{kappa}). Then $\nabla\kappa_i|_p=0$,
$i=1,2,3$.

 Choose local coordinates $x_1,...,x_5$ of $M$ on $U$ and define local coordinates of ${\Bbb T}$  by (\ref{coord}).

For $w\in\Lambda_{3}^2TM$ and $\tau\in{\Bbb T}$ with $\pi(w)=\pi(\tau)$, we set
$$
\widetilde{w}_{\tau}=w-\frac{1}{5}g(w,\tau)\tau,
$$
the ${\cal V}_{\tau}$-component of $w$.

 Any section $S$ of $\Lambda^2_3TM$ near the point $p$ yields a (local) vertical vector field $\tilde S$ on ${\Bbb T}$ defined by
$$
\widetilde S_{\tau}=S_{\pi(\tau)}-\frac{1}{5}g(S_{\pi(\tau)},\tau)\tau.
$$

We shall use this notation throughout this section and the following ones without further referring
to it.


\begin{lemma}\label{hor-ver} Let $X$ be a vector field on $M$ and $S$ a section of $\Lambda^2_3TM$  defined on a neighbourhood of the point $p=\pi(\sigma)$.
Then:
$$
\begin{array}{llc}


[X^h,\widetilde S]_{\sigma}=\widetilde{(\nabla_X S)}_{\sigma},\\[10pt]


[X^h,\Phi_{\pm}^{(n)}\widetilde S]_{\sigma}=\Phi_{\pm}^{(n)}\widetilde{(\nabla_X S)}_{\sigma},\\[10pt]


[\Phi_{\pm}^{(n)} X^h,\widetilde S]_{\sigma}=\widetilde{\Big(\nabla_{\displaystyle{\varphi_{\pm}^{\sigma}X_p}}S\Big)}_{\sigma}\\[6pt]
-\displaystyle{\frac{1}{2\pm 1}\{\pm 3[g(X_p,\xi_{\sigma})\widetilde{S}_{\sigma}(\xi_{\sigma}) -
g(\widetilde{S}_{\sigma}(\xi_{\sigma}), X_p)\xi_{\sigma}]
+[1\pm (-1)]\widetilde{S}_{\sigma}X\}_{\sigma}^h},\\[10pt]


[\Phi_{\pm}^{(n)} X^h,\Phi_{\pm}^{(n)}\widetilde S]_{\sigma}=\Phi_{\pm}^{(n)}\widetilde{\Big(\nabla_{\displaystyle{\varphi_{\pm}^{\sigma}}X_p}S\Big)}_{\sigma}\\[6pt]
-\displaystyle{\frac{1}{2\pm 1}\{\pm
3[g(X_p,\xi_{\sigma})(\Phi_{\pm}^{(n)}\widetilde{S}_{\sigma})(\xi_{\sigma})
-g((\Phi_{\pm}^{(n)}\widetilde{S}_{\sigma})(\xi_{\sigma}), X_p)\xi_{\sigma}]}\\[6pt]
\hfill+[1\pm (-1)](\Phi_{\pm}^{(n)}\widetilde{S}_{\sigma})X\}_{\sigma}^h,\\[10pt]


[(\Phi_{\pm}^{(n)})^2 X^h,\widetilde S]_{\sigma}=\widetilde{\Big(\nabla_{\displaystyle{(\varphi_{\pm}^{\sigma})^2X_p}}S\Big)}_{\sigma}\\[6pt]
-\displaystyle{\frac{1}{2\pm 1}\{\pm
3[g(X_p,\xi_{\sigma})\varphi_{\pm}^{\sigma}(\widetilde{S}_{\sigma}(\xi_{\sigma}))
+g(\varphi_{\pm}^{\sigma}(\widetilde{S}_{\sigma}(\xi_{\sigma})),X_p)\xi_{\sigma}]}\\[6pt]
\hfill+[1\pm (-1)][\varphi_{\pm}^{\sigma}(\widetilde{S}_{\sigma}X)+
\widetilde{S}_{\sigma}(\varphi_{\pm}^{\sigma}X)\}_{\sigma}^h,\\[10pt]


[(\Phi_{\pm}^{(n)})^2 X^h,\Phi_{\pm}^{(n)}\widetilde
S]_{\sigma}=\Phi_{\pm}^{(n)}\widetilde{\Big(\nabla_{\displaystyle{(\varphi_{\pm}^{\sigma})^2X_p}}S\Big)}_{\sigma}\\[6pt]
-\displaystyle{\frac{1}{2\pm 1}\{\pm
3[g(X_p,\xi_{\sigma})\varphi_{\pm}^{\sigma}((\Phi_{\pm}^{(n)}\widetilde{S}_{\sigma})(\xi_{\sigma}))
+g(\varphi_{\pm}^{\sigma}((\Phi_{\pm}^{(n)}\widetilde{S}_{\sigma})(\xi_{\sigma})),X_p)\xi_{\sigma}]}\\[6pt]
\hfill+[1\pm
(-1)][\varphi_{\pm}^{\sigma}((\Phi_{\pm}^{(n)}\widetilde{S}_{\sigma})X)+
(\Phi_{\pm}^{(n)}\widetilde{S}_{\sigma})(\varphi_{\pm}^{\sigma}X)\}_{\sigma}^h.\\[10pt]

\end{array}
$$
\end{lemma}

\begin{proof} In the local coordinates of ${\Bbb T}$ introduced above,
$$
\widetilde S=\sum_{i=1}^3\widetilde{S}_i\frac{\partial}{\partial y_i}\quad \mbox{ where }\quad
\widetilde{S}_i=\frac{1}{5}(g(S,\kappa_i)\circ\pi-y_i\sum_{j=1}^3y_jg(S,\kappa_j)\circ\pi).
$$
Note that, since $\nabla\kappa_i|_p=0$, $i=1,2,3$, we have
\begin{equation}\label{X}
\begin{array}{c}
X_{\sigma}^h=\sum_{\alpha=1}^5 X^{\alpha}(p)\Big(\displaystyle{\frac{\partial}{\partial \widetilde
x_{\alpha}}}\Big)_{\sigma},\quad [X^h,\displaystyle{\frac{\partial}{\partial y_i}}]_{\sigma}=0,\,
i=1,2,3,
\end{array}
\end{equation}
$$
(\nabla_X S)_p=\displaystyle{\frac{1}{5}}\sum_{i=1}^3X_p(g(S,\kappa_i))(\kappa_i)_p.
$$
Note also that, setting $S_i=\frac{1}{5}g(S,\kappa_i)\circ\pi$, we have
\begin{equation}\label{Phi-ver}
\Phi_{\pm}^{(n)}\widetilde S=(y_2S_3-y_3S_2)\frac{\partial}{\partial
y_1}+(y_3S_1-y_1S_3)\frac{\partial}{\partial y_2}+(y_1S_2-y_2S_1)\frac{\partial}{\partial y_3}.
\end{equation}

Now the first and the second formulas of the lemma follow from (\ref{h-lift}) and (\ref{Phi-ver}).

In view of (\ref{Phi}),(\ref{f}) and (\ref{X}), we have
$$
\begin{array}{l}
[\Phi_{\pm}^{(n)} X^h,\widetilde S]_{\sigma}=
\sum_{\alpha\beta=1}^5\sum_{i=1}^3 g_p(X,E_{\alpha})f^{\pm}_{\alpha \beta}(\sigma)(E_{\beta})^h_{\sigma}(\widetilde{S}_i)\Big (\displaystyle{\frac{\partial}{\partial
y_i}}\Big )_{\sigma}\\[10pt]
-\displaystyle{\frac{1}{2\pm 1}}\{\widetilde{S}_1(\sigma)[\pm 3g_p(\sqrt{3} E_1\wedge E_5,X\wedge E_{\beta})+[1\pm (-1)]g_p(\kappa_1,X\wedge E_{\beta})]\\[10pt]
+\widetilde{S}_2(\sigma)[\pm 3g_p(\sqrt{3} E_1\wedge E_3,X\wedge E_{\beta})+[1\pm (-1)]g_p(\kappa_2,X\wedge E_{\beta})]\}(E_{\beta})_{\sigma}^h \\[8pt]
=\widetilde{\Big(\nabla_{\displaystyle{\varphi_{\pm}^{\sigma}X_p}}S\Big)}_{\sigma}\\[8pt]
-\displaystyle{\frac{1}{2\pm 1}}\{\pm
3g_p(\xi_{\sigma}\wedge\widetilde{S}_{\sigma}(\xi_{\sigma}),X\wedge E_{\beta})
+[1\pm(-1)]g_p(\widetilde{S}_{\sigma},X\wedge E_{\beta})\}(E_{\beta})_{\sigma}^h
\end{array}
$$
since $(E_1)_p=-\xi_{\sigma}$ and $\sqrt{3}E_5=\kappa_1(E_1)$, $\sqrt{3} E_3=\kappa_2(E_1)$. This
implies the third formula. Using (\ref{Phi}) and (\ref{Phi-ver}) we obtain the forth formula by a
similar computation. Identities (\ref{Phi}) and (\ref{f}) imply that
$$
\begin{array}{l}
[(\Phi_{\pm}^{(n)})^2 X^h,\widetilde S]_{\sigma}=\widetilde{(D_{\displaystyle{(\varphi_{\pm}^{\sigma})^2X_p}}S)}_{\sigma}
-\displaystyle{\frac{1}{2\pm 1}}\{\pm 3[g_p(\xi_{\sigma}\wedge\widetilde{S}_{\sigma}(\xi_{\sigma}),X\wedge E_{\beta})\\[10pt]
\hskip 5.2cm +[1\pm(-1)]g_p(\widetilde{S}_{\sigma},X\wedge E_{\beta})]g_p(\varphi_{\pm}^{\sigma}E_{\beta},E_{\gamma})\\[8pt]
+[\pm 3g_p(\xi_{\sigma}\wedge\widetilde{S}_{\sigma}(\xi_{\sigma}),E_{\beta}\wedge E_{\gamma})\\[8pt]
\hskip 4cm+[1\pm(-1)]g_p(\widetilde{S}_{\sigma},E_{\beta}\wedge E_{\gamma})]g_p(\varphi_{\pm}^{\sigma}X,E_{\beta})\}(E_{\gamma})_{\sigma}^h.
\end{array}
$$
This gives the fifth formula of the lemma. The last one can be obtained by a similar computation.

\end{proof}

\begin{prop}
The contact distribution of $(\Phi_{\pm}^{(n)},\chi,h_t)$ is not integrable.
\end{prop}

\begin{proof}
The section $\kappa_1$ of $\Lambda^2_3TM$ yields the vertical vector field $\widetilde\kappa_1$ on
${\Bbb T}$ given by  $ \widetilde\kappa_1=(1-y_1^2)\frac{\partial}{\partial
y_1}-y_1y_2\frac{\partial}{\partial y_2}-y_1y_3\frac{\partial}{\partial y_3}. $ Put $A=\sqrt
3y_1y_3E_1^h-(\frac{1}{2}y_1^2+\frac{1}{2}y_2^2-y_3^2)E_3^h$. Then $\widetilde\kappa_1$ and $A$ are
vector fields on ${\Bbb T}$ perpendicular to the characteristic vector field $\chi$. Bearing in
mind identities (\ref{X}) and the fact that $\sigma=(\kappa_3)_{\pi(\sigma)}$, we easily see that
$[\widetilde\kappa_1,A]_{\sigma}=\sqrt 3(E_1)^h_{\sigma}=-\sqrt 3\chi_{\sigma}$. Thus
$\widetilde\kappa_1$ and $A$ are sections of the contact distributions whose Lie bracket is not a
section of it.
\end{proof}

\begin{lemma}\label{phi-hor-brac}
    Let $\sigma\in{\Bbb T}$ and let $X,Y$ be vector fields near the point $p=\pi(\sigma)$ such that $\nabla X|_p=0$, $\nabla Y|_p=0$. Then, for any $k,l=0,1,2$,
$$
\begin{array}{l}
[(\Phi_{\pm}^{(n)})^k X^{h},(\Phi_{\pm}^{(n)})^lY^{h}]_{\sigma}=-(T_p((\varphi_{\pm}^{\sigma})^k
X,(\varphi_{\pm}^{\sigma})^l Y))^h_{\sigma} + R_{p}((\varphi_{\pm}^{\sigma})^k
X,(\varphi_{\pm}^{\sigma})^l Y)\sigma,
\end{array}
$$
$n=1,2.$
\end{lemma}

\begin{proof} We have $\nabla\kappa_i|_p=\nabla(\ast\kappa_i)|_p=0$, $i=1,2,3$, since  $\nabla E_{\alpha}|_p=0$, $1\leq\alpha\leq 5$.
Then it is clear from (\ref{f}) and (\ref{X}) that $Z^h_{\sigma}(f_{\alpha \beta})=0$ for every
$Z\in T_pM$. In view of the assumption  $\nabla X|_p=\nabla Y|_p=0$,  an easy computation using
(\ref{bra-2h}) and (\ref{Phi}) gives the lemma.
\end{proof}

Denote the Levi-Civita connection of the metric $h_t$ by $D$. Let $\nabla^{LC}$ be the Levi-Civita
connection of the metric $g$ on $M$. Then the Koszul formula and the fact that the Lie bracket of a
vertical and a horizontal vector field is a vertical vector field imply
$$
(D_{X^h}Y^h)_{\sigma}=(\nabla^{LC}_{X}Y)^h_{\sigma}+\frac{1}{2}R(X,Y)\sigma.
$$
We have $\nabla_{X}Y=\nabla^{LC}_{X}Y+\frac{1}{2}T(X,Y)$ since $\nabla$ is a metric connection with
skew-symmetric torsion $T$. Thus
\begin{equation}\label{D-hh}
(D_{X^h}Y^h)_{\sigma}=(\nabla_{X}Y-\frac{1}{2}T(X,Y))^h_{\sigma}+\frac{1}{2}R(X,Y)\sigma.
\end{equation}
Let $V$ be a vertical vector field in a neighbourhood of a point $\sigma\in{\Bbb T}$. Take a
section $S$  of $\Lambda^2_3TM$  near the point $p=\pi(\sigma)$ such that $S_p=V_{\sigma}$ and
$\nabla S|_p=0$. We can also find sections $S_1$, $S_2$ such that $\nabla S_1|_p=\nabla S_2|_p=0$
and the corresponding vertical vector fields $\widetilde S_1$, $\widetilde S_2$ constitute a frame
of the vertical bundle ${\cal V}$ in a neighbourhood of $\sigma$. Then the Koszul formula and the
first identity of Lemma~\ref{hor-ver} imply that $(D_{\widetilde S}\widetilde S_k)_{\sigma}\in{\cal
V}_{\sigma}$. It follows that, for every vertical vector field $W$, $D_{V}W$ is a vertical vector
field. Thus the fibres of ${\Bbb T}$ are totally geodesic submanifolds. This also follows from the
Vilms theorem \cite{V} (or \cite{Besse}). Then $D_{V}X^h$ is orthogonal to every vertical vector
field, thus $D_{V}X^h$ is a horizontal vector field. Hence $D_{V}X^h={\cal H}D_{X^h}V$ since
$[V,X^h]$ is vertical. Therefore
\begin{equation}\label{D-vh}
\begin{array}{c}
g(D_{V}X^h,Y^h)_{\sigma}=g(D_{X^h}V,Y^h)_{\sigma}=-h_t(V,D_{X^h}Y^h)_{\sigma}=
-\displaystyle{\frac{t}{2}}g(R(X,Y)\sigma,V).
\end{array}
\end{equation}
\begin{lemma}\label{D-chi}
Let $V\in{\cal V}_{\sigma}$ and $X\in T_{\pi(\sigma)}M$. Then
$$
\begin{array}{c}
(D_{X^h}\chi)_{\sigma}=-\displaystyle{\frac{1}{2}(T(X,\xi_{\sigma}))^h_{\sigma}+\frac{1}{2}R(X,\xi_{\sigma})\sigma},\quad (D_{V}\chi)_{\sigma}\in{\cal
H}_{\sigma},\\[6pt]
 h_t(D_{V}\chi,X^h)_{\sigma}=g(({\cal J}V)(\xi_{\sigma}),X)+\displaystyle{\frac{t}{2}}g(R(X,\xi_{\sigma})\sigma,V).
\end{array}
$$
\end{lemma}

\begin{proof}
We have
\begin{equation}\label{chi}
\begin{array}{c}
\chi=\displaystyle{(\frac{1}{2}y_1^2+\frac{1}{2}y_2^2-y_3^2)E_1^h-\sqrt 3 y_1y_2E_2^h + \sqrt 3
y_1y_3E_3^h} -\displaystyle{\frac{\sqrt 3}{2}(y_1^2-y_2^2)E_4^h - \sqrt 3 y_2y_3E_5^h}
\end{array}
\end{equation}
in view of (\ref{xi}). Identities (\ref{X}), (\ref{chi}) and (\ref{D-hh}) imply
$$
\begin{array}{c}
(D_{X^h}\chi)_{\sigma}=-D_{X^h}E_1^h=\frac{1}{2}(T(X,E_1))^h_{\sigma}-\frac{1}{2}R(X,E_1)\sigma=\\[6pt]
-\frac{1}{2}(T(X,\xi_{\sigma}))^h_{\sigma}+\frac{1}{2}R(X,\xi_{\sigma})\sigma.
\end{array}
$$
Let $V=\displaystyle{v_1\Big(\frac{\partial}{\partial
y_1}\Big)_{\sigma}+v_2\Big(\frac{\partial}{\partial y_2}}\Big)_{\sigma}$. Then, by (\ref{chi}),
$$
\begin{array}{c}
(D_{V}\chi)_{\sigma}=-(D_{V}E_1^h)_{\sigma}+(\sqrt{3}v_1E_3-\sqrt{3}v_2E_5)^h_{\sigma}=\\[6pt]
-(D_{V}E_1^h)_{\sigma}+((v_1\kappa_2-v_2\kappa_1)(E_1))^h_{\sigma}=-(D_{V}E_1^h)_{\sigma}+((\sigma\times
V)(E_1))^h_{\sigma}.
\end{array}
$$
This and (\ref{D-vh}) imply the second identity of the lemma. Note also that $(D_{V}E_1^h)_{\sigma}\in{\cal H}_{\sigma}$, so
$(D_{V}\chi,X^h)_{\sigma}\in{\cal H}_{\sigma}$.
\end{proof}


\newpage

Set
$$
\eta_t (A)=h_t(A,\chi),~\Omega^{(n)}_{\pm\>t}(A,B)=h_t(A,\Phi_{\pm}^{(n)}B),~ A,B\in T{\Bbb T}.
$$
\begin{cor}\label{d-eta} If $V,W\in{\cal V}_{\sigma}$ and $X,Y\in T_{\pi(\sigma)}M$, then
$$
\begin{array}{c}
d\eta_t(X_{\sigma}^h+V,Y_{\sigma}^h+W)=-\displaystyle{\frac{1}{2}}[g(T(X,\xi_{\sigma}),Y)-g(T(Y,\xi_{\sigma}),X)]\\[6pt]
\pm [g(\varphi_{\pm}^{\sigma}X,W(\xi_{\sigma}))-g(\varphi_{\pm}^{\sigma}Y,V(\xi_{\sigma}))].$$
\end{array}
$$
In particular, $d\eta_t\neq 0$ at every point and the almost contact structure
$(\Phi^{(n)}_{\pm},\chi,h_t)$ is not contact, i.e. $d\eta_t\neq\Omega^{(n)}_{\pm\>t}$.
\end{cor}

\begin{proof} Note first that, for every $V\in{\cal V}_{\sigma}$,
\begin{equation}\label{comm}
({\cal J}_{\pm}^{(n)}V)(\xi_{\sigma})=\pm (-1)^{n+1} J_{\pm}^{\sigma}(V(\xi_{\sigma})), \quad
n=1,2.
\end{equation}
Indeed, it is easy to check this identity for $V=(\kappa_1)_p$ and $V=(\kappa_2)_p$ taking into
account the fact that, since $\sigma=(\kappa_3)_p$, we have $\xi_{\sigma}=-(E_1)_p$,
$H^{\sigma}=span\{E_2,E_4,E_3,E_5\}$,  $J_{\pm}^{\sigma}E_2=E_4$, $J_{\pm}^{\sigma}E_3=\pm E_5$.
By (\ref{comm}), we have
$$
({\cal J}V)(\xi_{\sigma})=(-1)^{n+1}(\Phi^{n}_{\pm}V)(\xi_{\sigma})=(-1)^{n+1}({\cal
J}_{\pm}^{(n)}V)(\xi_{\sigma})=\pm J_{\pm}^{\sigma}(V(\xi_{\sigma}))\\=\pm
\varphi_{\pm}^{\sigma}(V(\xi_{\sigma})).
$$
Now Lemma~\ref{D-chi} implies the formula for $d\eta_t$
stated in the corollary.

For $\sigma=\kappa_3$, $X=E_3$, $W=\kappa_1$, we have $d\eta_t(X,W)=\sqrt{3}$. Thus
$d\eta_t(X^h_{\sigma},W)=\pm g(E_5,-\sqrt 3 E_5)=\mp\sqrt 3$ while $\Omega^{(n)}_{\pm\>
t}(X^h_{\sigma},W)=0$.
\end{proof}
\begin{cor}\label{intcur}
Every integral curve of the characteristic vector field $\chi$ is a geodesic.
\end{cor}

Recall that a metric connection is said to be {\it characteristic} if its torsion
$T(X,Y,Z)\\=g(T(X,Y),Z)$ is totally skew-symmetric. According to \cite[Theorem 5.5]{BN} a
$SO(3)$-connection $\nabla$ on $M$ is characteristic if and only the tensor $\Upsilon$ satisfies
the identity $(\nabla^{LC}_v\Upsilon)(v,v,v)=0$ for every $v\in TM$ where $\nabla^{LC}$ is the
Levi-Civita connection of $(M,g)$. In this case we say that the $SO(3)$-structure is {\it nearly
integrable} by analogy with the case of a nearly K\"ahler structure.
The characteristic connection of such a structure is unique [ibid]. We refer to
\cite{ABF,BN,CF} for interesting example of nearly integrable structures.

\begin{cor}\label{Kill} If $\nabla$ is the characteristic connection, then:
\smallskip
\begin{enumerate}
\item[$(i)$] The vector field $\chi$ is conformally Killing only when it is Killing.
\smallskip
\item[$(ii)$] $\chi$ is Killing if and only if $\nabla$ is of constant curvature $t^{-1}$.
\end{enumerate}
\end{cor}
\begin{proof}
Suppose that ${\cal L}_{\chi}h_t=\lambda h_t$ for some constant $\lambda$, ${\cal L}$ being the Lie
derivative. According to Lemma~\ref{D-chi}, this identity is equivalent to the identities
$g(T(X,\xi_{\sigma}),Y)+g(T(Y,\xi_{\sigma}),X)=-2\lambda g(X,Y)$  and
$tg(R(X,\xi_{\sigma})\sigma,V)+g(({\cal J}V)(\xi_{\sigma}),X)=0$ for every $\sigma\in{\Bbb T}$,
$X,Y\in T_{\pi(\sigma)}M$, $V\in{\cal V}_{\sigma}$. The first of these identities holds if and only
if $\lambda=0$ since the torsion is skew-symmetric. Take an adapted basis $(a_1,...,a_5)$ of
$T_{\pi(\sigma)}M$. Then, in view of (\ref{xi}), the second identity is equivalent to
\begin{equation}\label{r}
tg(R(X,a_{\alpha})\sigma,V)+g({\cal J}V(a_{\alpha}),X)=0, \quad \alpha=1,...,5.
\end{equation}
Using the fact that $R$ is the curvature tensor of a metric connection, one can check that, for
every $\sigma,\tau\in\Lambda^2_3 T_pM$,
\begin{equation}\label{F}
g(R(X,Y)\sigma,\tau)=-g({\cal R}(X\wedge Y),\sigma\times\tau),\quad X,Y\in T_pM.
\end{equation}
It follows that identity (\ref{r}) is equivalent to $tg({\cal R}(X\wedge Y),{\cal J}V)=g({\cal J}V,X\wedge Y)$. The latter identity
holds if and only if $tg({\cal R}(X\wedge Y),\tau)=g(X\wedge Y,\tau)$  for every
$\tau\in\Lambda^2_3T_{\pi(\sigma)}M$, hence $t{\cal R}={\cal P}$ since ${\cal R}$ takes its values
in $\Lambda^2_3TM$.
\end{proof}

Lemma~\ref{hor-ver}, identity
(\ref{comm}) and Corollary~\ref{d-eta} imply the following.
\begin{prop}\label{NXV}
If $V\in{\cal V}_{\sigma}$, $\sigma\in{\Bbb T}$, and $X\in T_{\pi(\sigma)}M$, then
$$
\begin{array}{c}
N^{(n)}_{+}(X^h_{\sigma},V)=\{[(-1)^{n+1}-1]g(X,\xi_{\sigma})(J^{(n)}_{+}V)(\xi_{\sigma})+\\[6pt]
[(-1)^n+1]g((J^{(n)}_{+}V)(\xi_{\sigma}),X)\xi_{\sigma}\}_{\sigma}^h;\\[6pt]
N^{(n)}_{-}(X^h_{\sigma},V)=\{3[(-1)^{n}+1]g(X,\xi_{\sigma})(J^{(n)}_{-}V)(\xi_{\sigma})+\\[6pt]
[(-1)^n+3]g((J^{(n)}_{-}V)(\xi_{\sigma}),X)\xi_{\sigma}+2(J^{(n)}_{-}V)(X)-2\varphi^{\sigma}_{-}(VX)\}_{\sigma}^h.
\end{array}
$$
\end{prop}

Lemma~\ref{phi-hor-brac} and Corollary~\ref{d-eta} imply

\begin{prop}\label{NXY} Let $\sigma\in{\Bbb T}$ and $X,Y\in T_{\pi(\sigma)}M$. Then
$$
\begin{array}{l}
N^{(n)}_{\pm}(X^h_{\sigma},Y^h_{\sigma})=\{T(X,Y)-T(\varphi^{\sigma}_{\pm}X,\varphi^{\sigma}_{\pm}Y)
 +\varphi^{\sigma}_{\pm}T(\varphi^{\sigma}_{\pm}X,Y)+
\varphi^{\sigma}_{\pm}T(X,\varphi^{\sigma}_{\pm}Y)\}_{\sigma}^h\\[6pt]
-R(X,Y)\sigma+R(\varphi^{\sigma}_{\pm}X,\varphi^{\sigma}_{\pm}Y)\sigma -{\cal
J}^{(n)}_{\pm}R(\varphi^{\sigma}_{\pm}X,Y)\sigma-{\cal
J}^{(n)}_{\pm}R(X,\varphi^{\sigma}_{\pm}Y)\sigma.
\end{array}
$$
\end{prop}

\smallskip

\begin{prop} The almost contact structures $(\Phi^{(1)}_{-},\chi,h_t)$ and $(\Phi^{(2)}_{\pm},\chi,h_t)$
are not normal.
\end{prop}

\begin{proof} Take $\sigma=\kappa_3$ for an adapted basis $\{E_1,...,E_5\}$  and set
$V=(\kappa_1)_{\pi(\sigma)}$, $X=E_3$. We have $\xi_{\sigma}=-E_1$,
${\cal J}^{(n)}_{\pm}V=(-1)^n(\kappa_2)_{\pi(\sigma)}$,
$\varphi^{\sigma}_{-}E_2=E_4$, $\varphi^{\sigma}_{-}E_3=- E_5$.
Then, by the Proposition~\ref{NXV},
$N^{(1)}_{-}(E_3^h,\kappa_1)_{\sigma}=2(E_4)^h_{\sigma}$,
$N^{(2)}_{+}(E_3^h,\kappa_1)_{\sigma}=2\sqrt{3}(E_1)^h_{\sigma}$,
$N^{(2)}_{-}(E_3^h,\kappa_1)_{\sigma}=2(\sqrt{3}E_1+E_4)^h_{\sigma}$.
\end{proof}

In order to analyze the normality condition $N^{(1)}_+=0$, we need, in view of
Proposition~\ref{NXY}, some preliminaries on the curvature tensor $R$ of $\nabla$.

\subsection{A curvature decomposition} $\\$

Since $\nabla$ is a $SO(3)$-connection, the curvature operator at any point $p\in M$ takes its values in
$\Lambda^2_3T_pM$, i.e. ${\cal R}\in Hom(\Lambda^2TM,\Lambda^2_3TM)$.

The irreducible components of $Hom(\Lambda^2,\Lambda^2_3)$, where $\Lambda^2=\Lambda^2{\Bbb R}^5$,
under the action of $SO(3)$ have  been described in \cite[Proposition 5.8]{BN}. To state the result
of \cite{BN} we introduced the  following notation. Given a map ${\cal K}\in
Hom(\Lambda^2,\Lambda^2)$, we denote:

 - the corresponding $4$-tensor by $K$, i.e.
$$K(X,Y,Z,U)=g({\cal K}(X\wedge Y),Z\wedge U), \quad X,Y,Z,U\in {\Bbb R}^5,$$
where $g$ stands for the metric on $\Lambda^2$ induced by the standard metric of ${\Bbb R}^5$;

 - the Ricci tensor of $K$ -- by $\rho_K$, $$\rho_K(X,Y)=Trace\{(Z,U)\to K(X,Z,Y,U)\};$$

 - the symmetric and skew-symmetric parts of $\rho_K$ -- by $\rho_K^{+}$ and $\rho_K^{-}$,
 $$\rho_K^{\pm}(X,Y)=\frac{1}{2}(\rho_K(X,Y)\pm \rho_K(Y,X));$$

 - the scalar curvature of $K$ -- by $s_K$, $s_K=Trace\,\rho_K$;

 - the anti-symmetrization of $K$ -- by $A_K$.

 Then we have the following.

\begin{prop} \rm{(\cite{BN})}\label{irred} The map

$$
{\cal K}\to (A_K,\rho_K^{-},s_K g,\rho_K^{+}-\frac{1}{5}s_K g)
$$
is an equivariant isomorphism
$$
\Psi:Hom(\Lambda^2,\Lambda^2_3)\to
\Lambda^4\oplus(\Lambda^2_3\oplus\Lambda^2_7)\oplus\odot^2_1\oplus(\odot^2_5\oplus\odot^2_9).
$$
\end{prop}
Note that the summands on the right-hand side are irreducible under the $SO(3)$-action. We shall give
a description of their images in $Hom(\Lambda^2,\Lambda^2_3)$ under the inverse isomorphism
$\Psi^{-1}$.

  Denote by ${\cal P}: \Lambda^2=\Lambda^2_3\oplus\Lambda^2_7\to \Lambda^2_3$ the
orthogonal projection onto $\Lambda^2_3$. Given a form $\nu\in\Lambda^4$, let ${\cal
N}:\Lambda^2\to\Lambda^2$ be the associated symmetric linear map defined by $g({\cal N}(X\wedge
Y),Z\wedge U)=\nu(X\wedge Y\wedge Z\wedge U)$. Set
$$
{\cal K}_{\nu}=\frac{10}{3}{\cal P}\circ {\cal N}.
$$

\begin{lemma}\label{nu}
$\Psi({\cal K}_{\nu})=\nu$.
\end{lemma}

\begin{proof}
Let $\{E_1,...,E_5\}$ be an adapted basis of ${\Bbb R}^5$ and let $\{\kappa_1,\kappa_2,\kappa_3\}$
be the corresponding basis of $\Lambda^2_3$ defined via (\ref{kappa}).

Taking into account the skew-symmetry of $\nu$, one can easily see by a straightforward computation
that
$$
\rho_{ K_{\nu}}(X,Z)=\frac{2}{3}\sum_{\alpha=1}^5\sum_{i=1}^3\nu(X\wedge
E_{\alpha}\wedge\kappa_i)g(\kappa_i,Z\wedge E_{\alpha})=0.
$$
For the anti-symmetrization of $K_{\nu}$, we have
$$
\begin{array}{c}
A_{K_{\nu}}(X,Y,Z,U)=\\[6pt]\displaystyle{\frac{1}{9}}\sum_{\alpha=1}^5\sum_{i=1}^3 [\nu(X\wedge Y\wedge\kappa_i)g(\kappa_i,Z\wedge U)
+ \nu(Z\wedge U\wedge\kappa_i)g(\kappa_i,X\wedge Y)\\[8pt]
+\nu(Y\wedge Z\wedge\kappa_i)g(\kappa_i,X\wedge U)+\nu(X\wedge U\wedge\kappa_i)g(\kappa_i,Y\wedge Z) \\[8pt]
+\nu(Z\wedge X\wedge\kappa_i)g(\kappa_i,Y\wedge U) + \nu(Y\wedge U\wedge\kappa_i)g(\kappa_i,Z\wedge
X)].
\end{array}
$$
It is easy to check by means of this formula that $(AK_{\nu})=\nu$ on the standard basis of
$\Lambda^4$ yielded by $\{E_1,...,E_5\}$. Thus $(AK_{\nu})=\nu$ on $\Lambda^4$ and the lemma is
proved.
\end{proof}

Now, let $\eta\in\otimes^2{\Bbb R}^5$ and denote again by $\eta$ the endomorphism of ${\Bbb R}^5$
determined by the form $\eta$ via $g(\eta(X),Y)=\eta(X,Y)$. The form $\eta $ is the Ricci tensor of
the operator $\Lambda^2\to\Lambda^2$ defined by $X\wedge Y\to \eta(X)\wedge Y+X\wedge\eta(Y)$.
Since we need an operator with values in $\Lambda^2_3$, we consider ${\cal P}(\eta(X)\wedge
Y+X\wedge\eta(Y))$. It is not hard to check that  the anti-symmetrization of the corresponding
$4$-tensor vanishes on the standard basis of $\Lambda^4$ yielded by an adapted basis of ${\Bbb
R}^5$. To find the Ricci tensor, we note first that a direct computation gives
$$
\sum_{i=1}^3g(\kappa_iX,\kappa_iY)=\sum_{i=1}^3\sum_{\alpha=1}^5g(X,\kappa_iE_{\alpha})g(Y,\kappa_iE_{\alpha})=6g(X,Y).
$$

Then
$$
\begin{array}{c}
\sum_{\alpha=1}^5\sum_{i=1}^3g( \eta(X)\wedge E_{\alpha},\kappa_i)g(\kappa_i,Z\wedge
E_{\alpha})=\sum_{i=1}^3g(\kappa_i(\eta(X)),\kappa_i(Z))\\[6pt]
=6g(\eta(X),Z)=6\eta(X,Z).
\end{array}
$$
Moreover, in the cases when $\eta$ is ether symmetric or skew-symmetric, we have
$$
\begin{array}{l}
\sum_{\alpha=1}^5\sum_{i=1}^3g(X\wedge\eta(E_{\alpha}),\kappa_i)g(\kappa_i,Z\wedge
E_{\alpha})\\[6pt]
=\sum_{\alpha=1}^5\sum_{i=1}^3g(\kappa_iX,\eta(E_{\alpha})g(\kappa_iZ,E_{\alpha})=\pm\sum_{i=1}^3\eta(\kappa_iX,\kappa_iZ)
\end{array}
$$
where the plus sign corresponds to the case when $\eta$ is symmetric and the minus sign-- to the
case of a skew-symmetric form $\eta$. This suggests to define a (symmetric, respectively,
skew-symmetric) form setting
\begin{equation}\label{eta-prim}
\eta'(X,Z)=Trace\{\Lambda^2_3\ni\sigma\to \eta(\sigma X,\sigma Z)\}.
\end{equation}

Set
$$
{\cal K}^{-}_{\eta}(X\wedge Y)=\frac{5}{6}{\cal P}(\eta(X)\wedge Y+X\wedge\eta(Y)+\eta'(X)\wedge
Y+X\wedge\eta'(Y)).
$$
\begin{lemma}\label{skew}
If $\eta\in \Lambda^2$, then  $\Psi({\cal K}^{-}_{\eta})=\eta$.
\end{lemma}

\begin{proof}
Since $\eta$ is skew-symmetric, $\eta'(X\wedge Z)=\eta(\sum_{i=1}^3 \kappa_iX\wedge\kappa_iZ)$ and
one can see by a direct computation that
$$
\sum_{i=1}^2\kappa_i X\wedge \kappa_i Z=\sum_{i=1}^3g(\kappa_i X,Z)\kappa_i=\sum_{i=1}^3g(X\wedge
Z,\kappa_i)\kappa_i=5{\cal P}(X\wedge Z).
$$
Thus, $\eta'=5\eta\circ{\cal P}$. Then the anti-symmetrization of the $4$-tensor $K^{-}_{\eta}$
vanishes and its Ricci tensor is equal to $\eta$.
\end{proof}

Now set
$$
{\cal K}^{+}_{\eta}(X\wedge Y)=\frac{5}{18}{\cal P}(5\eta(X)\wedge Y+5X\wedge\eta(Y)-\eta'(X)\wedge
Y-X\wedge\eta'(Y))
$$
\begin{lemma}\label{symm}
In the notation (\ref{ort}), if $\eta\in\odot^2_5\oplus\odot^2_9$, then $\Psi({\cal
K}^{+}_{\eta})=\eta$.
\end{lemma}
\begin{proof}
We set $\eta''=(\eta')'$ for any symmetric form $\eta$. Then a simple (but tedious) computation
shows that
$$
\eta'(X,Z)+\eta''(X,Z)=12\eta(X,Y)+6Trace\,\eta . g(X,Z).
$$
Noting that the last term vanishes when $\eta$ is orthogonal to $g$, we get the lemma.
\end{proof}

\smallskip

The preceding considerations give also the following.
\begin{lemma}\label{P}
$\Psi(\displaystyle{\frac{5}{6}}{\cal P})=g$.
\end{lemma}

\subsection{Conditions for normality of $(\Phi^{(1)}_{+},\chi,h_t)$} $\\$

Now, in order to give necessary and sufficient conditions for normality of the almost contact metric
structure $(\Phi^{(1)}_{+},\chi,h_t)$, we define a tensor on $M$ as follows. The definition
$\xi_{\sigma}=\frac{1}{4}\ast(\sigma\wedge\sigma)$ makes sense for every $\sigma\in\Lambda^2_3TM$.
Note that we have $g(\xi_{\sigma},\xi_{\sigma})=|\sigma|^4$. We extend also the definition of the
operator $\varphi^{\sigma}_{+}$ setting
$\varphi^{\sigma}_{+}X=\frac{1}{3}\imath_{X}(\sigma|\sigma|^2+\imath_{\xi_{\sigma}}(\ast\sigma))$
for $\sigma\in\Lambda^2_3TM$ and $X\in T_{\pi(\sigma)}M$. As above, let $A_R$ be the
anti-symmetrization of the curvature $4$-tensor $R$ and $\rho^{+}_R$ the symmetric part of the
Ricci tensor. Set
$$
\eta=\rho_R^{+}-\displaystyle{\frac{1}{5}}s_R g
$$
and
\begin{equation}\label{tenQ}
\begin{array}{c}
Q(\sigma,X)=12[A_R(\xi_{\sigma}\wedge\varphi^{\sigma}_{+}X\wedge\sigma)
-A_R(\xi_{\sigma}\wedge\imath_X\sigma\wedge\sigma)|\sigma|^2]\\[6pt]
+5[g(X,\xi_{\sigma})\eta(\xi_{\sigma},\xi_{\sigma})-
\eta(X,\xi_{\sigma})g(\xi_{\sigma},\xi_{\sigma})-\eta(\varphi^{\sigma}_{+}(\imath_X\sigma),\xi_{\sigma})].
\end{array}
\end{equation}

Let $(\kappa_1,\kappa_2,\kappa_3)$ be the frame of $\Lambda^2_3TM$ defined by means of an adapted
frame $(a_1,...,a_5)$ of $TM$ via formulas (\ref{kappa}). Then, for
$\sigma=y_1\kappa_1+y_2\kappa_2+y_3\kappa_3$, the coefficients of $\xi_{\sigma}$ in the basis
$(a_1,...,a_5)$ are homogeneous polynomials of degree $2$ in $y_1,y_2,y_3$ while the coefficients
of $\varphi^{\sigma}_{+}X$, $X$-fixed, in the basis $a_i\wedge a_j$ are homogeneous polynomials in
$y_1,y_2,y_3$ of degree $3$. Thus $Q(\sigma,X)$ is a homogeneous polynomial of degree $6$ in
$\sigma$ and is linear in $X$. Therefore $Q(\sigma,X)$ determined a section of the bundle
$\odot^6\Lambda^2_3TM\otimes TM$. The latter is a subbundle of $\otimes^{13} TM$, so $Q$ can be
considered as a section of $\otimes^{13} TM$, i.e. as a tensor on $M$. This tensor will again be
denoted by $Q$.

We need the following.

\begin{lemma}\label{Q}
We have:
$$
\begin{array}{c}
Q(\sigma,\xi_{\sigma})=0,\\[6pt]
Q(\sigma,X)=0 \mbox { for every } X\in{\cal V}(\xi_{\sigma})=\{V(\xi_{\sigma}): V\in{\cal
V}_{\sigma}\}, \\[6pt]
Q(\sigma,X)=-6A_R(\xi_{\sigma}\wedge
\imath_X\sigma\wedge\sigma)|\sigma|+5\eta(X,\xi_{\sigma})g(\xi_{\sigma},\xi_{\sigma})~ \mbox {for
}X\in ({\Bbb R}\xi_{\sigma}\oplus {\cal V}(\xi_{\sigma}))^{\perp}.
\end{array}
$$
\end{lemma}

\begin{proof} The identity $Q(\sigma,\xi_{\sigma})=0$ is obvious
in view of the identities $\varphi^{\sigma}_{+}(\xi_{\sigma})=0$ and
$\imath_{\xi_{\sigma}}\sigma=0$. Since $Q(\sigma,X)$ is homogeneous in $\sigma$, to prove the remaining
identities of the lemma, we may assume that $|\sigma|^2=5$. Then, by Lemma~\ref{ort-basis}, we can
take an adapted basis $a=(a_1,...,a_5)$ of $T_{\pi(\sigma)}M$ such that $\sigma=\kappa_3(a)$.
Recall that in this case $\xi_{\sigma}=-a_1$, $H^{\sigma}=span\{a_2,a_4,a_3,a_5\}$,
$\varphi^{\sigma}_{+}a_2=a_4$, $\varphi^{\sigma}_{+}a_3=a_5$. Moreover, ${\cal
V}_{\sigma}=span\{\kappa_1,\kappa_2\}$, so ${\cal V}(\xi_{\sigma})=span\{a_3,a_5\}$ and $({\Bbb
R}\xi_{\sigma}\oplus {\cal V}(\xi_{\sigma}))^{\perp}=span\{a_2,a_4\}$.  Thus
$\imath_X\sigma=\varphi^{\sigma}_{+}X$ for $X\in{\cal V}(\xi_{\sigma})$,
$\imath_X\sigma=2\varphi^{\sigma}_{+}X$ for $X\in ({\Bbb R}\xi_{\sigma}\oplus {\cal
V}(\xi_{\sigma}))^{\perp}$, and the result follows.
\end{proof}

\begin{theorem} \label{normal} Suppose that the $SO(3)$-structure on $M$ is nearly integrable.
Then the almost contact structure $(\Phi^{(1)}_{+},\chi,h_t)$ is
normal if and only if $\ast T\in\Lambda^2_3TM$, the tensor $Q$ and
the $\odot^2_9$, and $\Lambda^2_7$-components of ${\cal R}$ vanish.
\end{theorem}

\begin{proof}
Let $V$, $W$ be vertical vectors at a point $\sigma\in{\Bbb T}$ and $X,Y\in T_{\pi(\sigma)}M$. The
restriction of $\Phi^{(1)}_{+}$ to the vertical bundle is a complex structure, so, in view of
Corollary~\ref{d-eta}, $N^{(1)}_{+}(V,W)=0$. We also have $N^{(1)}_{+}(X^h_{\sigma},V)=0$ by
Proposition~\ref{NXV}. Therefore, in view of Proposition~\ref{NXY}, the almost contact structure
$(\Phi^{(1)}_{+},\chi,h_t)$ is normal if and only if
\begin{equation}\label{TXY}
T(X,Y)-T(\varphi^{\sigma}_{+}X,\varphi^{\sigma}_{+}Y)\\[6pt]
+\varphi^{\sigma}_{+}T(\varphi^{\sigma}_{+}X,Y)+ \varphi^{\sigma}_{+}T(X,\varphi^{\sigma}_{+}Y)=0
\end{equation}
and
$$
R(X,Y)\sigma-R(\varphi^{\sigma}_{+}X,\varphi^{\sigma}_{+}Y)\sigma
+{\cal J}^{(1)}_{+}R(\varphi^{\sigma}_{+}X,Y)\sigma+{\cal
J}^{(1)}_{+}R(X,\varphi^{\sigma}_{+}Y)\sigma=0.
$$
In view of (\ref{F}), the latter identity is is equivalent to
\begin{equation}\label{RXY}
g({\cal R}(-X\wedge Y+\varphi^{\sigma}_{+}X\wedge \varphi^{\sigma}_{+}Y),\sigma\times\tau)+g({\cal R}(\varphi^{\sigma}_{+}X\wedge
Y+X\wedge\varphi^{\sigma}_{+}Y),\tau)=0
\end{equation}
for every $\sigma, \tau\in {\Bbb T}$ with $\pi(\sigma)=\pi(\tau)$, $\sigma\perp\tau$ and every
$X,Y\in T_{\pi(\sigma)}M$.

\smallskip

Identity (\ref{TXY}) is equivalent to  vanishing of the following tensor
$$
\begin{array}{c}
\widetilde T(X,Y,Z)=g(T(X,Y),Z)-g(T(\varphi^{\sigma}_{+}X,\varphi^{\sigma}_{+}Y),Z)\\[6pt]
-g(T(\varphi^{\sigma}_{+}X,Y),\varphi^{\sigma}_{+}Z)-g(T(X,\varphi^{\sigma}_{+}Y),\varphi^{\sigma}_{+}Z)
\end{array}
$$
Take an adapted basis $(a_1,...,a_5)$ of $T_pM$, $p=\pi(\sigma)$,
such that $\sigma=\kappa_3$, so $\xi_{\sigma}=-a_1$,
$\varphi^{\sigma}_{+}a_2=a_4$, $\varphi^{\sigma}_{+}a_3=a_5$. Since
the torsion is totally skew-symmetric, the tensor $\widetilde T$ is
also skew-symmetric, hence $\widetilde T=0$ exactly when $\widetilde
T(a_i,a_j,a_k)=0$ for $1\leq i<j<k\leq 5$. Direct computation shows
that these identities hold if and only if
$g(T(a_1,a_2),a_3)-g(T(a_1,a_4),a_5)=0$ and
$g(T(a_1,a_2),a_5)+g(T(a_1,a_4),a_2)=0$. The latter identities are
equivalent to
\begin{equation}\label{T}
g(T(J^{\sigma}_{+}X,J^{\sigma}_{+}Y),\xi_{\sigma})=g(T(X,Y),\xi_{\sigma})
\end{equation}
for every $X,Y\in H^{\sigma}$.

Suppose that the
torsion $T$ satisfies this condition. Take a point $p\in M$ and an adapted basis $\{a_1,...,a_5\}$
of $T_pM$. Define $\{\kappa_1,\kappa_2,\kappa_3\}$ by means of (\ref{kappa}). Set
$T_{ijk}=g(T(a_i,a_j),a_k)$, $i,j,k=1,...,5$. For $\sigma=\kappa_1$, we have
$\xi_{\sigma}=\frac{1}{2}a_1-\frac{\sqrt 3}{2}a_4$ by (\ref{xi}), hence
$\frac{2}{3}\sigma_{+}=(\frac{\sqrt 3}{2}a_1+\frac{1}{2}a_4)\wedge a_5+a_2\wedge a_3$ by
(\ref{sigma-pm}). Thus, in this case, $J^{\sigma}_{+}(\frac{\sqrt 3}{2}a_1+\frac{1}{2}a_4)=a_5$ and
$J^{\sigma}_{+}a_2=a_3$. Then identity (\ref{T}) with $X=\frac{\sqrt 3}{2}a_1+\frac{1}{2}a_4$, $Y=a_2$
and $X=\frac{\sqrt 3}{2}a_1+\frac{1}{2}a_4$, $Y=a_3$ gives
\begin{equation}\label{T1}
T_{135}+\sqrt{3}T_{345}=2T_{124},\quad T_{125}+\sqrt{3}T_{245}=-2T_{134}.
\end{equation}
If $\sigma=\kappa_2$, then $\xi_{\sigma}=\frac{1}{2}a_1+\frac{\sqrt 3}{2}a_4$  and
$\frac{2}{3}\sigma_{+}=(\frac{\sqrt 3}{2}a_1-\frac{1}{2}a_4)\wedge a_3+a_2\wedge a_5$. It
follows from (\ref{T}) that
\begin{equation}\label{T2}
T_{135}-\sqrt{3}T_{345}=2T_{124},\quad T_{123}+\sqrt{3}T_{234}=-2T_{145}.
\end{equation}
For $\sigma=\kappa_3$, we have $\xi_{\sigma}=-a_1$, $\frac{2}{3}\sigma_{+}=a_2\wedge a_4+a_3\wedge
a_5$ and by (\ref{T})
\begin{equation}\label{T3}
T_{145}=T_{123},\quad T_{134}=T_{125}.
\end{equation}
Moreover, for $\sigma=\frac{1}{\sqrt 2}(\kappa_1+\kappa_2)$, we have
$\xi_{\sigma}=\frac{1}{2}a_1-\frac{\sqrt 3}{2}a_2$, hence $\frac{2}{3}\sigma_{+}=(\frac{\sqrt
6}{4}a_1+\frac{\sqrt 2}{4}a_2+\frac{2\sqrt 2}{4}a_4)\wedge a_5+(\frac{\sqrt 6}{4}a_1+\frac{\sqrt
2}{4}a_2-\frac{2\sqrt 2}{4}a_4)\wedge a_3$. In this case identity (\ref{T}) gives
\begin{equation}\label{T4}
2T_{124}=T_{135}-\sqrt{3}T_{235}.
\end{equation}
The first identities of (\ref{T1}) and (\ref{T2}) imply
\begin{equation}\label{T5}
T_{345}=0,\quad 2T_{124}-T_{135}=0.
\end{equation}
This and (\ref{T4}) give
\begin{equation}\label{T6}
T_{235}=0.
\end{equation}
The second identities of (\ref{T1}) and (\ref{T2}), and identity (\ref{T3}) imply
\begin{equation}\label{T7}
\sqrt{3}T_{134}+T_{245}=0,\quad \sqrt{3}T_{145}+T_{234}=0.
\end{equation}
Thus, according to (\ref{T3}), (\ref{T5}) -- (\ref{T7}), the $2$-form $\ast T$ vanishes on the
following basis of  $(\Lambda^2_3 T_pM)^{\perp}$:~~$a_2\wedge a_3-a_4\wedge a_5$, $a_2\wedge
a_5-a_3\wedge a_4$, $a_1\wedge a_2$, $a_2\wedge a_4-2a_3\wedge a_5$, $a_1\wedge a_4$, $a_1\wedge
a_3-\sqrt{3}a_2\wedge a_5$,  $a_1\wedge a_5-\sqrt{3}a_2\wedge a_3$. Therefore $\ast T\in
\Lambda^2_3 T_pM$.

 Conversely, assume that the latter condition holds for every $p\in M$. Take $\sigma\in{\Bbb T}$
and choose and adapted basis $\{a_1,...,a_5\}$ of $T_{p}M$, $p=\pi(\sigma)$, such that
$\sigma=\kappa_3$. Then, at the point p, identity (\ref{T}) is equivalent to identities (\ref{T3}).
Obviously, these identities are satisfied  since $\ast T$ vanishes on $a_2\wedge a_3-a_4\wedge a_5$
and $a_2\wedge a_5-a_3\wedge a_4$.

\smallskip

Now we address the problem when identity (\ref{RXY}) holds for every $\sigma, \tau\in {\Bbb T}$ with $\pi(\sigma)=\pi(\tau)$, $\sigma\perp\tau$ and
every $X,Y\in H^{\sigma}$.

In the notation of Proposition~\ref{irred}, denote the component of ${\cal R}$
in $\Psi^{-1}(\Lambda^4)$, $\Psi^{-1}(\Lambda^2_3\oplus\Lambda^2_7)$,
$\Psi^{-1}(\odot^2_5\oplus\odot^2_9)$ by ${\cal A}$, ${\cal B}^{-}$, ${\cal B}^{+}$, respectively.
Then, in view of Lemma~\ref{P},
\begin{equation}\label{R-decom}
{\cal R}=\frac{5s_R}{6}{\cal P}+{\cal A} + {\cal B}^{-} + {\cal B}^{+}.
\end{equation}

Take $\sigma\in{\Bbb T}$ and choose an adapted basis
$\{a_1,...,a_5\}$ of $T_pM$, $p=\pi(\sigma)$, such that
$\sigma=\kappa_3$. Then $H^{\sigma}=span\{a_2,a_4,a_3,a_5\}$ and
$\varphi^{\sigma}_{+}a_2=a_4$, $\varphi^{\sigma}_{+}a_3=a_5$.

It is easy to check that for every $X,Y\in H^{\sigma}$ and $ \tau\in {\Bbb T}$, $\tau\perp\sigma$,
we have $-X\wedge Y+\varphi^{\sigma}_{+}X\wedge \varphi^{\sigma}_{+}Y\perp\tau$. This implies
$$
g({\cal P}(-X\wedge Y+\varphi^{\sigma}_{+}X\wedge \varphi^{\sigma}_{+}Y),\sigma\times\tau)+g({\cal P}(\varphi^{\sigma}_{+}X\wedge
Y+X\wedge \varphi^{\sigma}_{+}Y),\tau)=0.
$$
Next, in view of Lemma~\ref{nu} (with $\nu=A_R$), it is easy  to check that
$$
g({\cal A}(-X\wedge Y+\varphi^{\sigma}_{+}X\wedge \varphi^{\sigma}_{+}Y),\sigma\times\tau)+g({\cal A}(\varphi^{\sigma}_{+}X\wedge
Y+X\wedge \varphi^{\sigma}_{+}Y),\tau)=0.
$$
Now, let $\eta$ be a bilinear form on $T_pM$. For $X,Y\in H^{\sigma}$ and $\tau\in{\Bbb T}$ with
$\pi(\tau)=\pi(\sigma)$, set
\begin{equation}\label{DE}
\begin{array}{c}
{\cal E}_{\eta}(X,Y,\sigma,\tau)=\\[6pt]g(-\eta(X)\wedge Y-X\wedge\eta(Y)
+\eta(\varphi^{\sigma}_{+}X)\wedge \varphi^{\sigma}_{+}Y+\varphi^{\sigma}_{+}X\wedge
 \eta(\varphi^{\sigma}_{+}Y),\sigma\times\tau)\\[6pt]
+g(\eta(\varphi^{\sigma}_{+}X)\wedge Y+\varphi^{\sigma}_{+}X\wedge\eta(Y)+\eta(X)\wedge \varphi^{\sigma}_{+}Y+X\wedge
\eta(\varphi^{\sigma}_{+}Y),\tau).
\end{array}
\end{equation}
We have
\begin{equation}\label{E}
\begin{array}{c}
{\cal E}_{\eta}(X,Y,\sigma,\tau)=-{\cal E}_{\eta}(Y,X,\sigma,\tau), \> {\cal
E}_{\eta}(\varphi^{\sigma}_{+}X,\varphi^{\sigma}_{+}Y,\sigma,\tau)=-{\cal E}_{\eta}(X,Y,\sigma,\tau),\\[6pt]
{\cal E}_{\eta}(X,\varphi^{\sigma}_{+}X,\sigma,\tau)=0,\\[6pt]
{\cal E}_{\eta}(a_2,a_3,\kappa_3,\kappa_1)={\cal E}_{\eta}(a_2,a_5,\kappa_3,\kappa_2)=2(\eta(a_2,a_4)+\eta(a_4,a_2)),\\[6pt]
{\cal E}_{\eta}(a_2,a_3,\kappa_3,\kappa_2)=-{\cal
E}_{\eta}(a_2,a_5,\kappa_3,\kappa_1)=2(\eta(a_2,a_2)-\eta(a_4,a_4)).
\end{array}
\end{equation}
In view of Lemma~\ref{skew} (with $\eta=\rho_R^{-}$), this implies that, for $\tau\perp\sigma$,
$X,Y\in H^{\sigma}$,
$$
g({\cal B}^{-}(-X\wedge Y+\varphi^{\sigma}_{+}X\wedge \varphi^{\sigma}_{+}Y),\sigma\times\tau)+g({\cal
B}^{-}(\varphi^{\sigma}_{+}X\wedge Y+X\wedge \varphi^{\sigma}_{+}Y),\tau)=0.
$$
If $\eta$ is symmetric, then
$$
\eta(a_2,a_4)=-4\eta'(a_2,a_4),\quad \eta(a_2,a_2)-\eta(a_4,a_4)=-4(\eta'(a_2,a_2)-\eta'(a_4,a_4)),
$$
where, as above, $\eta'(X,Z)=Trace\{\Lambda^2_3\ni\sigma\to \eta(\sigma X,\sigma Z)\}$. The latter
identities, (\ref{E}) and Lemma~\ref{symm} imply that for $\eta\in\odot^2_5 T_pM\oplus\odot^2_9
T_pM$
\begin{equation}\label{Rmin}
\begin{array}{c}
g({\cal K}^{+}_{\eta}(-X\wedge Y+\varphi^{\sigma}_{+}X\wedge \varphi^{\sigma}_{+}Y),\sigma\times\tau)+g({\cal
K}^{+}_{\eta}(\varphi^{\sigma}_{+}X\wedge Y+X\wedge \varphi^{\sigma}_{+}Y),\tau)\\[6pt]
=5{\cal E}_{\eta}(X,Y,\sigma,\tau).
\end{array}
\end{equation}
It follows that identity (\ref{RXY}) holds for every $\sigma, \tau\in {\Bbb T}$ with $\pi(\sigma)=\pi(\tau)$, $\sigma\perp\tau$ and
for every $X,Y\in H^{\sigma}$  if and only if
$$
{\cal E}_{\eta}(X,Y,\sigma,\tau)=0
$$
where $\eta=\rho_R^{+}-\displaystyle{\frac{1}{5}}s_R g$.

Suppose  identity (\ref{RXY}) holds for  $\sigma, \tau\in {\Bbb T}$ with $\pi(\sigma)=\pi(\tau)$, $\sigma\perp\tau$ and
$X,Y\in H^{\sigma}$. Let $p\in M$ and let $\{a_1,...,a_5\}$ be an adapted basis of $T_pM$.
Define $\{\kappa_1,\kappa_2,\kappa_3\}$ by means of this basis. Set
$\eta_{\alpha\beta}=\eta(a_{\alpha},a_{\beta})$ for $\eta=\rho_R^{+}-\displaystyle{\frac{1}{5}}s_R
g$, $\alpha,\beta=1,...,5$. Then (\ref{E}) gives
$$
\eta_{24}=0,\quad \eta_{22}-\eta_{44}=0.
$$
If $\sigma=\kappa_2$, then $\varphi^{\sigma}_{+}(\frac{\sqrt 3}{2}a_1-\frac{1}{2}a_4)=a_3$,
$\varphi^{\sigma}_{+}a_2=a_5$ and (\ref{DE}) implies
$$
\sqrt{3}\eta_{14}-\frac{3}{2}\eta_{11}+2\eta_{33}-\frac{1}{2}\eta_{44}=0,\quad
\sqrt{3}\eta_{13}-\eta_{34}=0.
$$
For $\sigma=\kappa_1$, we have $\varphi^{\sigma}_{+}(\frac{\sqrt 3}{2}a_1+\frac{1}{2}a_4)=a_5$, $\varphi^{\sigma}_{+}a_2=a_3$
and we obtain from (\ref{DE}) that
$$
\sqrt{3}\eta_{14}+\frac{3}{2}\eta_{11}-2\eta_{55}+\frac{1}{2}\eta_{44}=0,\quad
\sqrt{3}\eta_{15}+\eta_{45}=0.
$$

Every adapted basis $(a_1,...,a_5)$ can be used to obtain a new adapted basis $(a_1',...,a'_5)$ by
means of formulas (\ref{a'}).  According to the preceding considerations,
$\sqrt{3}\eta(a_1',a_3')-\eta(a_3',a_4')=0$. In view of (\ref{a'}), this identity can be written as
$$
\begin{array}{c}
(\sqrt{3}\eta_{13}+2\eta_{25}\sin^2\psi)\cos\psi
+(\sqrt{3}\eta_{15}+2\eta_{23}\cos^2\psi)\sin\psi \\[6pt]
-(\eta_{34}\cos\psi+\eta_{45}\sin\psi)\cos 2\psi=0.
\end{array}
$$
Taking $\psi=\displaystyle{\frac{\pi}{4}}$ and $\psi=\displaystyle{\frac{3\pi}{4}}$, we obtain
$$
(\sqrt{3}\eta_{13}+\eta_{25}) +(\sqrt{3}\eta_{15}+\eta_{23})=0,\quad -(\sqrt{3}\eta_{13}+\eta_{25})
+(\sqrt{3}\eta_{15}+\eta_{23})=0
$$
Thus, we have
$$
\sqrt{3}\eta_{13}+\eta_{25}=0,\quad \sqrt{3}\eta_{15}+\eta_{23}=0.
$$
Finally, consider the adapted basis $\{c_1,...,c_5\}$ obtained from $\{a_1,...,a_5\}$ by means of
formulas (\ref{c}). Then $\eta(c_2,c_2)=\eta(c_4,c_4)$ and we obtain the identity
$$
\begin{array}{c}
-\displaystyle{\frac{3}{4}\eta_{11}-\frac{1}{4}\eta_{22}\sin^2 2\theta+\eta_{33}\cos^2\theta+\eta_{55}\sin^2\theta}\\[6pt]
-(\displaystyle{\frac{\sqrt{3}}{2}}\eta_{12}-\eta_{35})\sin 2\theta
-\displaystyle{(\frac{\sqrt{3}}{2}\eta_{14}+\frac{1}{2}\eta_{24}\sin
2\theta-\frac{1}{4}\eta_{44}\cos 2\theta)\cos 2\theta}=0.
\end{array}
$$
Taking $\theta=\displaystyle{\frac{\pi}{4}}$ and $\theta=\displaystyle{\frac{3\pi}{4}}$, we get
$$
-\displaystyle{\frac{3}{4}\eta_{11}-\frac{1}{4}\eta_{22}+\frac{1}{2}\eta_{33}+\frac{1}{2}\eta_{55}}=0,\quad
\sqrt{3}\eta_{12}-2\eta_{35}=0.
$$
Thus, the form $\eta=\rho_R^{+}-\displaystyle{\frac{1}{5}}s_R g$ vanishes on the vectors $a_2\odot
a_4$, $a_2\odot a_2-a_4\odot a_4$, $2\sqrt{3}a_1\odot a_4-3a_1\odot a_1 +4a_3\odot a_3-a_4\odot
a_4$, $\sqrt{3}a_1\odot a_3-a_3\odot a_4$, $2\sqrt{3}a_1\odot a_4+3a_1\odot a_1 -4a_5\odot
a_5+a_4\odot a_4$, $\sqrt{3}a_1\odot a_5+a_4\odot a_5$, $\sqrt{3}a_1\odot a_3+a_2\odot a_5$,
$\sqrt{3}a_1\odot a_5+a_2\odot a_3$, $3a_1\odot a_1+a_2\odot a_2-2a_3\odot a_3-2a_5\odot a_5$,
$\sqrt{3}a_1\odot a_2-2a_3\odot a_5$. This vectors constitute a basis of $\odot^2_9T_pM$. It
follows that the $\odot^2_9$-component of ${\cal R}$ vanishes.

  Conversely, suppose that $\odot^2_9$-component of ${\cal R}$ vanishes. Let $\sigma\in{\Bbb T}$ and take an adapted basis $\{a_1,...,a_5\}$ of $T_pM$, $p=\pi(\sigma)$,
such that $\sigma=\kappa_3$. By assumption, the form
$\eta=\rho_R^{+}-\displaystyle{\frac{1}{5}}s_R g$ vanishes on the vectors $a_2\odot a_4$ and
$a_2\odot a_2-a_4\odot a_4$. Hence, by (\ref{E}), ${\cal E}_{\eta}(X,Y,\sigma,\tau)=0$ for
every $\sigma, \tau\in {\Bbb T}$ with $\pi(\sigma)=\pi(\tau)$, $\sigma\perp\tau$ and $X,Y\in
H^{\sigma}$. Then, by (\ref{Rmin}),
$$
g({\cal K}^{+}_{\eta}(-X\wedge Y+\varphi^{\sigma}_{+}X\wedge
\varphi^{\sigma}_{+}Y),\sigma\times\tau)+g({\cal K}^{+}_{\eta}(\varphi_{+}^{\sigma}X\wedge Y+X\wedge
\varphi^{\sigma}_{+}Y),\tau)=0.
$$
Now, it follows from the preceding considerations that
identity (\ref{RXY}) holds for $\sigma, \tau\in {\Bbb T}$ with $\pi(\sigma)=\pi(\tau)$, $\sigma\perp\tau$ and
$X,Y\in H^{\sigma}$.

\smallskip

Next, note that, in view of (\ref{F}), identity (\ref{RXY}) holds for $\sigma\in{\Bbb
T}$, $X\in H^{\sigma}$ and $Y=\xi_{\sigma}$ if and only if
\begin{equation}\label{RXxi}
g({\cal R}(X\wedge\xi_{\sigma}),\sigma\times\tau)-g({\cal R}(\varphi^{\sigma}_{+}X\wedge\xi_{\sigma}),\tau)=0
\end{equation}
for every $\sigma\in{\Bbb T}$, $\tau\in{\cal V}_{\sigma}$, $X\in
H^{\sigma}$. Let $p\in M$, $\sigma\in \Lambda^2_3T_pM$, $|\sigma|^2=5$. Taking an
adapted basis $a=(a_1,...,a_5)$ of $T_pM$ such that $\sigma=\kappa_3(a)$, it is easy to check that
the operator ${\cal P}$ satisfies identity (\ref{RXxi}). Suppose that this identity holds. Let
$a=(a_1,...,a_5)$ be an arbitrary adapted basis of $T_pM$. Set
$\rho^{-}_{\alpha\beta}=\rho_R^{-}(a_{\alpha},a_{\beta })$,
$\eta=\rho_R^{+}-\displaystyle{\frac{1}{5}}s_R g$ and
$\eta_{\alpha\beta}=\eta(a_{\alpha},a_{\beta})$, $\alpha,\beta=1,...,5$. Recall that for
$\sigma=\kappa_1(a)$, we have $\xi_{\sigma}=\frac{1}{2}a_1-\frac{\sqrt 3}{2}a_4$, $X=\frac{\sqrt
3}{2}a_1+\frac{1}{2}a_4\in H^{\sigma}$ and $\varphi^{\sigma}_{+}X=a_5$. Set $\tau=\kappa_2(a)$. Then, by a computation using
(\ref{R-decom}),  we get from (\ref{RXxi})  that
$$
3\sqrt 3((\sqrt
3\eta_{12}-2\eta_{35})+\eta_{24})+3\rho^{-}_{12}-\sqrt
3(\rho^{-}_{24}-2\rho^{-}_{35})=0.
$$
We have $\sqrt 3\eta_{12}-2\eta_{35}=0$ and $\eta_{24}=0$ since the $\odot^2_9$-component of ${\cal
R}$ vanishes, i.e $\eta$ vanishes on $\odot^2_9TM$. Hence
\begin{equation}\label{r-1}
3\rho^{-}_{12}-\sqrt 3(\rho^{-}_{24}-2\rho^{-}_{35})=0.
\end{equation}
Similarly, applying (\ref{RXxi}) with $\sigma=\kappa_1(a)$, $\tau=\kappa_3(a)$, $X=\frac{\sqrt
3}{2}a_1+\frac{1}{2}a_4$ and taking into account (\ref{R-decom}) and the fact that
$\eta|\odot^2_9TM=0$, we get
\begin{equation}\label{r-2}
(\rho^{-}_{13}-\sqrt 3\rho^{-}_{34})+2(\rho^{-}_{13}-\sqrt
3\rho^{-}_{25})=0.
\end{equation}
If $\sigma=\kappa_2(a)$, then
$\xi_{\sigma}=\frac{1}{2}a_1+\frac{\sqrt 3}{2}a_4$, $X=\frac{\sqrt
3}{2}a_1-\frac{1}{2}a_4\in H^{\sigma}$ and $\varphi^{\sigma}_{+}X=a_3$.
Then, putting $\tau=\kappa_1(a)$ and $\tau=\kappa_3(a)$ in (\ref{RXxi}), we
obtain the identities
\begin{equation}\label{r-3}
3\rho^{-}_{12}+\sqrt 3(\rho^{-}_{24}-2\rho^{-}_{35})=0,\quad
(\rho^{-}_{15}-\sqrt 3\rho^{-}_{45})+2(\rho^{-}_{15}-\sqrt
3\rho^{-}_{23})=0.
\end{equation}
For $\sigma=\kappa_3(a)$, $X=a_2$, identity (\ref{RXxi}) with
$\tau=\kappa_1(a)$ and $\tau=\kappa_2(a)$ gives
\begin{equation}\label{r-4}
\rho^{-}_{23}-\rho^{-}_{45}=0,\quad \rho^{-}_{25}-\rho^{-}_{34}=0.
\end{equation}
We also set $\sigma=\frac{1}{\sqrt 2}(\kappa_1(a)+\kappa_2(a))$. Then
$\xi_{\sigma}=\frac{1}{2}a_1-\frac{\sqrt 3}{2}a_2$, $X=\frac{\sqrt 3}{2}a_1+\frac{1}{2}a_2\in
H^{\sigma}$ and $\varphi^{\sigma}_{+}X=\frac{1}{\sqrt 2}(a_3+a_5)$. Setting $\tau=\kappa_1(a)-\kappa_2(a)$ in
(\ref{RXxi}), we obtain
\begin{equation}\label{r-5}
\sqrt 3\rho^{-}_{14}+(\rho^{-}_{24}-2\rho^{-}_{35})=0.
\end{equation}
Identities (\ref{r-1}) - (\ref{r-5}) imply that $\rho_R^{-}$ vanishes on the bi-vectors $a_1\wedge
a_2$, $a_2\wedge a_4-2a_3\wedge a_5$, $a_1\wedge a_4$, $a_2\wedge a_3 -a_4\wedge a_5$, $a_2\wedge
a_5-a_3\wedge a_4$, $a_1\wedge a_5-\sqrt 3 a_4\wedge a_5$, $a_1\wedge a_3-\sqrt 3 a_3\wedge a_4$.
These bi-vectors form a basis of the space $\Lambda^2_7T_pM$. Hence $\rho_R^{-}|\Lambda^2_7T_pM=0$.
This means that $\Lambda^2_7$-component of ${\cal R}$ vanishes. Now, let $\sigma\in{\Bbb T}$ be
arbitrary. Take an adapted basis $a=(a_1,...,a_5)$ of $T_{\pi(\sigma)}M$ such that
$\sigma=\kappa_3(a)$ and apply (\ref{RXxi}) with $X=a_3$ and $\tau=\kappa_1(a)$, Then a computation
making use of Lemmas~\ref{nu}--\ref{symm} gives
$$
24 A_{1235}+\rho_{14}^{-} + (-4\sqrt 3\eta_{33}+4\sqrt 3\eta_{55}-2\eta_{14})=0
$$
where $A_{1235}=A_R(a_1\wedge a_2\wedge a_3\wedge a_5)$. We have $\rho_{14}^{-}=0$ since
$\rho_R^{-}|\Lambda^2_7TM=0$ and $-4\sqrt 3\eta_{33}+4\sqrt 3\eta_{55}-2\eta_{14}=10\eta_{14}$
since $\eta|\odot^2_9TM=0$. Thus
$$
12A_R(a_1\wedge a_2\wedge a_3\wedge a_5)+5\eta(a_1,a_4)=0
$$
Applying (\ref{RXxi}) with $X=a_3$ and $\tau=\kappa_2(a)$ we get by a similar computation that
$$
12A_R(a_1\wedge a_3\wedge a_4\wedge a_5)+5\eta(a_1,a_2)=0.
$$
The left-hand sides of the last two identities are equal to $-Q(\kappa_3,a_4)$ and
$-Q(\kappa_3,a_2)$, respectively. It follows that $Q(\sigma,X)=0$ for $X\perp ({\cal
V}(\xi_{\sigma})\oplus {\Bbb R}\xi_{\sigma})$. This and Lemma~\ref{Q} imply that $Q(\sigma,X)=0$
for every $\sigma\in{\Bbb T}$, $X\in T_{\pi(\sigma)}M$. It follows that $Q=0$ in view of the
homogeneity of $Q(\sigma;X)$ in $\sigma$.

Conversely, suppose that $\ast T\in\Lambda^2_3TM$, $Q=0$ and the $\odot^2_9$, and
$\Lambda^2_7$-components of ${\cal R}$ vanish.  Fix
an arbitrary $\sigma\in{\Bbb T}$ and take an adapted basis of $T_{\pi(\sigma)}M$ for which
$\sigma=\kappa_3$. Then it is easy to check that identities (\ref{TXY}) and (\ref{RXY}) are fulfilled.
\end{proof}

\smallskip

The next statement, combined with Lemma~\ref{Q},  might be useful to check the condition $Q=0$.

\begin{lemma}
Let $A$ be a skew-symmetric $4$-form and $\eta$ a symmetric bilinear form on ${\Bbb R}^5$ such that
$\eta|\odot^2_9=0$. Define a tensor $Q$ by (\ref{tenQ}) with $A_R=A$.

Suppose that $Q(q_i,X)=0$, $i=1,2,3$, for an orthogonal basis $q_1,q_2,q_3$ of $\Lambda^2_3$ with
$|q_1|^2=|q_2|^2=|q_3|^2=5$ and every  $X\in {\Bbb R}^5$.

Then $Q(\sigma,X)=0$ for every $\sigma\in \Lambda^2_3$ and every $X\in {\Bbb R}^5$.
\end{lemma}

\begin{proof}
It is enough to show that $Q(\sigma,X)=0$ in the case when $|\sigma|=5$. We may assume that the
basis $q_1,q_2,q_3$ yields the orientation of $\Lambda^2_3$. Then, by Lemma~\ref{ort-basis}, there
are adapted bases $a=(a_1,...,a_5)$ and $a'=(a_1',...,a_5')$ of ${\Bbb R}^5$ such that
$q_1=\kappa_1(a), q_2=\kappa_2(a), q_3=\kappa_3(a)$ and $\sigma=\kappa_1(a')$. Since the group
$SO(3)$ acts transitively on the set of adapted bases, there is a matrix $h$ in $SO(3)$ such that
$\imath (h)$ sends the basis $a$ to the basis $a'$. We shall use the notation in the proof of
Lemma~\ref{orient}. First we show that $Q(\tau,X)=0$ for $\tau=\kappa_1(b),\kappa_2(b),\kappa_3(b)$
where $\kappa_i(b)$, $i=1,2,3$, are given by (\ref{k(b)}) and the adapted basis $b=(b_1,...,b_5)$
is defined by (\ref{b}). If $\tau=\kappa_1(b)$, then $\xi_{\tau}=\frac{1}{2}b_1-\frac{\sqrt
3}{2}b_4=\frac{1}{2}a_1-\frac{\sqrt 3}{2}a_4\sin 2\varphi+\frac{\sqrt 3}{2}a_4\cos 2\varphi$ and
$({\Bbb R}\xi_{\tau}\oplus{\cal V}(\xi_{\tau}))^{\perp}=span\{\frac{\sqrt
3}{2}b_1+\frac{1}{2}b_4,b_5\}$ where $\frac{\sqrt 3}{2}b_1+\frac{1}{2}b_4=-\frac{\sqrt
3}{2}a_1-\frac{1}{2}a_2\sin 2\varphi+\frac{1}{2}a_4\cos 2\varphi$ and
$b_5=a_3\cos\varphi+a_5\sin\varphi$. Set $A_{\alpha\beta\gamma\delta}=A(a_{\alpha}\wedge
a_{\beta}\wedge a_{\gamma}\wedge a_{\delta})$ and $\eta_{\alpha\beta}=\eta(a_{\alpha},a_{\beta})$,
$\alpha,\beta,\gamma,\delta=1,...,5$. Then, applying Lemma~\ref{Q}, we obtain
$$
\begin{array}{c}
Q(\kappa_1(b),\frac{\sqrt 3}{2}b_1+\frac{1}{2}b_4)=\displaystyle{\frac{\sin
2\varphi}{2}}(12A_{1345}+5\eta_{12})-\displaystyle{\frac{\cos 2\varphi}{2}}(12A_{1235}+5\eta_{14})\\[6pt]
-\sqrt 3(6A_{2345}+\displaystyle{\frac{5}{4}}\eta_{11}-\displaystyle{\frac{5}{4}}\eta_{22}\sin^2
2\varphi -\displaystyle{\frac{5}{4}}\eta_{44}\cos^2 2\varphi)-\displaystyle{\frac{5\sqrt
3}{4}}\eta_{24}\sin 4\varphi .
\end{array}
$$
We also have
$$
12A_{1345}+5\eta_{12}=-Q(\kappa_3(a),a_2)=0,
$$
$$
12A_{1235}+5\eta_{14}=-Q(\kappa_3(a),a_4)=0,
$$
$$
-6A_{1235}+6\sqrt{3}
A_{2345}+\displaystyle{\frac{5}{4}}(\sqrt{3}\eta_{11}-2\eta_{14}-\sqrt{3}\eta_{44})=Q(\kappa_1(a),\frac{\sqrt{3}}{2}a_1+\frac{1}{2}a_4)=0.
$$
The last two identities imply
$$
6A_{2345}+\displaystyle\frac{5}{4}(\eta_{11}-\eta_{44})=0.
$$
We have $\eta_{22}=\eta_{44}$ and $\eta_{24}=0$ since $\eta|\odot^2_9=0$. It follows that
$$Q(\kappa_1(b),\frac{\sqrt 3}{2}b_1+\frac{1}{2}b_4)=0.$$

It is easy to compute that
$$
\begin{array}{c}
Q(\kappa_1(b),b_5)=Q(\kappa_1(a),a_5)\sin\varphi+ Q(\kappa_2(a),a_3)\cos\varphi\\[6pt]
+\displaystyle{\frac{5\sqrt 3}{2}}(\eta_{45}-\eta_{23})\sin 2\varphi\cos\varphi -
\displaystyle{\frac{5\sqrt 3}{2}}(\eta_{25}+\eta_{34})\sin 2\varphi\sin\varphi.
\end{array}
$$
By assumption, $Q(\kappa_1(a),a_5)=Q(\kappa_2(a),a_3)=0$ and $\eta_{45}-\eta_{23}=0$,
$\eta_{25}+\eta_{34}=0$. Thus, $$Q(\kappa_1(b),b_5)=0.$$

Similar computations show also that
$$
Q(\kappa_2(b),\frac{\sqrt
3}{2}b_1-\frac{1}{2}b_4)=Q(\kappa_2(b),b_3)=0,
$$
$$
Q(\kappa_3(b),b_2)=Q(\kappa_3(b),b_4)=0.
$$
Now Lemma~\ref{Q} implies that $Q(\tau,X)=0$ for $\tau=\kappa_1(b),\kappa_2(b),\kappa_3(b)$ and
$X\in{\Bbb R}^5$. Using this, (\ref{c}) and (\ref{k(c)}) we see in a similar way that $Q(\tau,X)=0$
for $\tau=\kappa_1(c),\kappa_2(c),\kappa_3(c)$. The latter identity, (\ref{a'}) and (\ref{k(a')})
imply $Q(\kappa_1(a'),X)=0$.
\end{proof}

\smallskip
\noindent {\bf Examples}. 1. According to \cite[Theorem 4.7]{BN} every $5$-manifold $M$ with an
irreducible $SO(3)$-structure whose characteristic connection coincides with the Levi-Civita
connection is locally isometric to one of the following symmetric space ${\Bbb R}^5$,
$SU(3)/SO(3)$, $SL(3,{\Bbb R})/SO(3)$. The Riemannian metric of these spaces is Einstein. It is
flat only for $M={\Bbb R}^5$; in the other two cases it is not even conformally flat [ibid.].
Clearly we have $\rho_R^{-}=0$, $\rho_R^{+}-\displaystyle{\frac{1}{5}}s_R g=0$. Also, by the
algebraic Bianchi identity, $A_R=0$, hence $Q=0$. Thus, by Theorem~\ref{normal}, the almost
contact structure of the twistor space of $M$ is normal.

\smallskip

\noindent 2. Consider $SO(2)$ as a subgroup of $SO(3)$ via the standard diagonal imbedding $A\to
(1,A)$. Then $SO(2)$ can also be considered as a subgroup of $SO(3)\times SO(1,2)$ by means of the
map $SO(2)\ni A\to (A^2,A)$. It is shown in the proof of \cite[ Proposition 6.3 ]{BN} that for
every $t\in{\Bbb R}\setminus\{0\}$ there exists linearly independent invariant $1$-forms
$\tilde\theta_1,..., \tilde\theta_5$ such that:

\smallskip
\noindent $(1)$ the tensors
$$
\tilde g={\tilde\theta_1}^2+....+{\widetilde\theta_5}^2,\quad
\tilde\Upsilon=
\displaystyle{\frac{1}{2}}\tilde\theta_1(6{\tilde\theta_2}^2+6{\tilde\theta}_4^2-2\tilde{\theta}_1^2-3{\tilde\theta_3}^2-3{\tilde\theta_5}^2)
+\displaystyle{\frac{3\sqrt 3}{2}}\tilde\theta_4({\tilde\theta_5}^2-{\tilde\theta_3}^2) +3\sqrt
3\tilde\theta_2\tilde\theta_3\tilde\theta_5
$$
descend to a nearly integrable irreducible
$SO(3)$-structure $(g,\Upsilon)$ on the $5$-dimensional manifold $M=(SO(3)\times SO(1,2))/SO(2)$,

\smallskip
\noindent $(2)$ if $\theta_1,....,\theta_5$ is the co-frame on $M$ induced by the forms
$\tilde\theta_1,..., \tilde\theta_5$, the torsion $3$-form $T$ and the curvature tensor $R$ of the
characteristic connection are given by
$$
\begin{array}{c}
T=t(\theta_1\wedge\theta_2\wedge\theta_4+2\theta_1\wedge\theta_3\wedge\theta_5),\\[6pt]
g(R(X,Y)Z,U)=2t^2g(X\wedge Y,\kappa_3)g(\kappa_3,Z\wedge U) \quad X,Y,Z,U\in TM,
\end{array}
$$
where $\kappa_3$ is defined by means of the dual frame $E_1,...,E_5$ of $\theta_1,....,\theta_5$
({\it note}: the curvature tensor used here differs by a sign from that used in \cite{BN}).

We refer also to \cite{ABF} for more information about the irreducible $SO(3)$-structure on
$M=(SO(3)\times SO(1,2))/SO(2)$.

Clearly $\ast T=-t(2\theta_2\wedge\theta_4+\theta_3\wedge\theta_5)\in\Lambda^2_3T^{\ast}M$. For the
Ricci tensor $\rho$ we have the following identities
$$
\begin{array}{c}
\rho(E_1,E_{\alpha})=0,\quad \rho(E_2,E_{\alpha})=8t^2\delta_{2\alpha},\quad
\rho(E_3,E_{\alpha})=2t^2\delta_{3\alpha},\\[6pt]
\rho(E_4,E_{\alpha})=8t^2\delta_{4\alpha}, \quad \rho(E_5,E_{\alpha})=2t^2\delta_{5\alpha},\quad
\alpha=1,...,5.
\end{array}
$$
Hence $\rho^{-}=0$, in particular the $\Lambda^2_7$-component of ${\cal R}$ vanishes. Set
$\eta=\rho^{+}-\displaystyle{\frac{1}{5}}s g$ where $s$ is the scalar curvature. Then
$$
\begin{array}{c}
\eta(E_1,E_{\alpha})=-4t^2\delta_{1\alpha},\quad \eta(E_2,E_{\alpha})=4t^2\delta_{2\alpha},\quad
\eta(E_3,E_{\alpha})=-2t^2\delta_{3\alpha},\\[6pt]
\eta(E_4,E_{\alpha})=4t^2\delta_{4\alpha}, \quad \eta(E_5,E_{\alpha})=-2t^2\delta_{5\alpha},\quad
\alpha=1,...,5.
\end{array}
$$
This implies that $\eta$ vanishes on the basis of $\odot^2_9TM$ we have used in the
proof of Theorem~\ref{normal}, thus the $\odot^2_9$-component of ${\cal R}$ vanishes. If $A_R$ is
the anti-symmetrization of the curvature tensor, set $A_{ijkl}=A_R(E_i\wedge E_j\wedge E_k\wedge
E_l)$. Then $A_{1235}=0$, $A_{2345}=-\frac{4t^2}{3}$, thus
$$
\begin{array}{c}
Q(\kappa_1,\frac{\sqrt 3}{2}E_1+\frac{1}{2}E_4)=\\[6pt]
-6A_{1235}+6\sqrt{3}A_{2345}+\frac{5}{4}(\sqrt{3}\eta_{11}-2\eta_{14}-\sqrt{3}\eta_{44})=
-18\sqrt{3}t^2\neq 0.
\end{array}
$$
Therefore, by Theorem~\ref{normal}, the almost contact metric structure $(\Phi^{(1)}_{+},\chi,h_t)$
is not normal.

\section{Addendum}

In this section we give a new proof of the integrability result of \cite{BN} for the almost $CR$-structures $({\cal D},\cal
J^{(n)}_{\pm})$ on the manifold ${\Bbb T}$ defined at the beginning of Section 3.

   Recall that an almost  Cauchy-Riemann ($CR$) structure
on a manifold $N$ is a  pair $({\cal D},{\cal J})$ of a subbundle ${\cal D}$ of the tangent bundle
$TN$ and an almost complex structure ${\cal J}$ of the bundle ${\cal D}$. For any two sections
$X,Y$ of ${\cal D}$, the value of $[X,Y] \> mod\, {\cal D}$ at a point $p\in N$ depends only on the
values of $X$ and $Y$ at $p$, so we have a skew-symmetric bilinear form $\omega: {\cal D}\times
{\cal D}\to TN/{\cal D}$ defined by $\omega(X,Y)=[X,Y] \> mod\, {\cal D}$ and called the Levi form
of the almost $CR$-structure $({\cal D},{\cal J})$. If the Levi form is ${\cal J}$-invariant, we can
define the Nijenhuis tensor of the almost $CR$-structure $({\cal D},{\cal J})$ by
$$
N^{\it CR}(X,Y)=-[X,Y]+[{\cal J}X,{\cal J}Y]-{\cal J}([{\cal J}X,Y]+[X,{\cal J}Y]);
$$
its value at a point $p\in N$ lies in ${\cal D}$ and depends only on the values of the sections
$X,Y$ at $p$. An almost $CR$-structure is said to be {\it integrable} if its Levi form is ${\cal J}$-invariant
and the Nijenhuis tensor vanishes. An integrable almost $CR$-structure is called simpy a $CR$-structure.
Let ${\cal D}^{\Bbb C}={\cal D}^{1,0}\oplus {\cal D}^{0,1}$ be
the decomposition of the complexification of ${\cal D}$ into $(1,0)$ and $(0,1)$ parts with respect
to ${\cal J}$. If the almost $CR$-structure $({\cal D},{\cal J})$ is integrable, then the bundle ${\cal
D}^{1,0}$ satisfies the following two conditions:
$$
{\cal D}^{1,0}\cap \overline{{\cal D}^{1,0}}=0, \quad [\Gamma({\cal D}^{1,0}),\Gamma({\cal
D}^{1,0})]\subset\Gamma({\cal D}^{1,0})
$$
where $\Gamma({\cal D}^{1,0})$ stands for the space of smooth sections of ${\cal D}^{1,0}$.
Conversely, suppose we are given a complex subbundle ${\cal E}$ of the complexified tangent bundle
$T^{\Bbb C}N$ such that ${\cal E}\cap\bar {\cal E}=0$ and $[\Gamma({\cal E}),\Gamma({\cal
E})]\subset\Gamma({\cal E})$ (a bundle with these properties is often also called a
"$CR$-structure"). Set ${\cal D}=\{X\in TN: X=Z+\bar Z \> \mbox{ for some (unique)} \> Z\in {\cal
E}\}$ and put ${\cal J}X=-{\it Im}\,Z$ for $X\in {\cal D}$. Then $({\cal D},{\cal J})$ is an
integrable almost $CR$-structure such that ${\cal D}^{1,0}={\cal E}$.

\subsection{The Levi form of the almost $CR$-structures on the twistor space}

\begin{lemma}\label{Levi}
    Let $A,B\in {\cal D}_{\sigma}$ be horizontal vectors and $V,W\in{\cal D}_{\sigma}$ vertical ones at a
point $\sigma\in{\Bbb T}$. Then the Levi form $\omega_{\pm}^{(n)}$ of the almost  $CR$-structure $({\cal D},\cal
J^{(n)}_{\pm})$ is given by:
$$
\begin{array}{l}
\omega_{\pm}^{(n)}(A,B)=-g_p(T(\pi_{\ast}A,\pi_{\ast}B),\xi_{\sigma})(\xi_{\sigma})_{\sigma}^h,\\[8pt]
\omega_{\pm}^{(n)}(V,W)=0,\\[8pt]
\omega_{\pm}^{(n)}(A,V)=\pm
g_p(\pi_{\ast}A,J_{\pm}^{\sigma}(V(\xi_{\sigma}))(\xi_{\sigma})_{\sigma}^h.
\end{array}
$$
\end{lemma}

\begin{proof} Take vector fields $X$ and $Y$ near the point
$p=\pi(\sigma)$ such that $X_p=\pi_{*}A$, $Y_p=\pi_{*}B$ and $\nabla X|_p=\nabla Y|_p=0$. Then
$-(\Phi_{\pm}^{(n)})^2X^{h}$ and $-(\Phi_{\pm}^{(n)})^2Y^{h}$ are sections of ${\cal D}$ with
values $A$ and $B$ at the point $\sigma$, so
$$
\omega_{\pm}^{(n)}(A,B)=((\Phi_{\pm}^{(n)})^2+Id)[(\Phi_{\pm}^{(n)})^2X^{h},
(\Phi_{\pm}^{(n)})^2Y^h]_{\sigma}.
$$
The tangent vectors $X_p,Y_p$  are orthogonal to $\xi_{\sigma}$ and Lemma~\ref{phi-hor-brac} implies
$$
\begin{array}{c}
\omega_{\pm}^{(n)}(A,B)=-((\varphi_{\pm}^{\sigma})^2(T_p(X,Y)))_{\sigma}^h-(T_p(X,Y))_{\sigma}^h=
-g_p(T(X,Y),\xi_{\sigma})(\xi_{\sigma})_{\sigma}^h.
\end{array}
$$

\smallskip

Extend $V$ and $W$ to vertical vector fields of ${\Bbb T}$ on a neighbourhood of ${\sigma}$. These
vector fields are section of ${\cal D}$, whose Lie bracket is a vertical vector field, hence a
section of  ${\cal D}$. Therefore
$$
\omega_{\pm}^{(n)}(V,W)=0.
$$

    Finally, take a section $S$ of $\Lambda^2_3TM$ near the point $p$ such that $S_p=V$ and $\nabla S|_p=0$.
Then, by Lemma~\ref{hor-ver}, we have
$$
\begin{array}{l}
\omega_{\pm}^{(n)}(A,V)=((\Phi_{\pm}^{(n)})^2+Id)([-(\Phi_{\pm}^{(n)})^2X^h,\widetilde S]_{\sigma})=\\[8pt]
\displaystyle{\frac{1}{2\pm 1}}\{\pm 3g_p(\varphi_{\pm}^{\sigma}(V(\xi_{\sigma})),X)(\xi_{\sigma})_{\sigma}^h\\[8pt]
\hskip 3.1cm +[1\pm (-1)][(\varphi_{\pm}^{\sigma})^2 (V\circ \varphi_{\pm}^{\sigma}(X))+V\circ\varphi_{\pm}^{\sigma}(X)]_{\sigma}^h\}=\\[8pt]
\displaystyle{\frac{1}{2\pm 1}}\{\pm 3g_p(\varphi_{\pm}^{\sigma}\circ V(\xi_{\sigma}),X) + [1\pm
(-1)]g_p(V\circ\varphi_{\pm}^{\sigma}(X),\xi_{\sigma})\}(\xi_{\sigma})_{\sigma}^h=\\[8pt]
\displaystyle{\frac{1}{2\pm 1}}\{\pm 3 + [1\pm (-1)]\}g_p(\varphi_{\pm}^{\sigma}\circ
V(\xi_{\sigma}),X)(\xi_{\sigma})_{\sigma}^h.
\end{array}
$$
Note that $V(\xi_{\sigma})$ is orthogonal to $\xi_{\sigma}$ since the map $V$ is skew-symmetric.
Hence $\varphi_{\pm}^{\sigma}(V(\xi_{\sigma}))=J_{\pm}^{\sigma}(V(\xi_{\sigma}))$. This implies the
third formula of the lemma.

\end{proof}

\begin{cor}\label{Levi-pm-n}

(i)~ The Levi form $\omega_{+}^{(1)}$, respectively  $\omega_{-}^{(2)}$, is ${\cal
J}_{+}^{(1)}$-invariant, respectively ${\cal J}_{-}^{(2)}$-invariant, if and only if
$$
g(T(J_{\pm}^{\sigma}X,J_{\pm}^{\sigma}Y),\xi_{\sigma})=g(T(X,Y),\xi_{\sigma})
$$
for every $\sigma\in{\Bbb T}$ and $X,Y\in T_{\pi(\sigma)}M$, $X,Y\perp\xi_{\sigma}$.

\smallskip

(ii)~ The Levi form $\omega_{-}^{(1)}$ is not $ {\cal J}_{-}^{(1)}$-invariant and  $\omega_{+}^{(2)}$ is not ${\cal J}_{+}^{(2)}$-invariant.
\end{cor}

\begin{proof}
By  Lemma~\ref{Levi} and identity (\ref{comm}), if $A\in{\cal D}_{\sigma}$ is a horizontal vector and $V\in{\cal
V}_{\sigma}$,
$$
\omega_{\pm}^{(n)}({\cal J}_{\pm}^{(n)}A, {\cal J}_{\pm}^{(n)}V)=\pm(-1)^{n+1}\omega_{\pm}^{(n)}(A,V).
$$
This and Lemma~\ref{Levi} imply $(i)$. Also,  $\omega_{+}^{(2)}$ is ${\cal J}_{+}^{(2)}$-invariant
if and only if $\omega_{\pm}^{(2)}(A,V)=0$ for every $A$ and $V$. By Lemma~\ref{Levi}, we have
$\omega_{+}^{(2)}((E_3)_{\sigma}^h,(\kappa_1)_p)$ = $-\sqrt{3}$. Thus $\omega_{+}^{(2)}$ is not
${\cal J}_{+}^{(2)}$-invariant. Similarly, $\omega_{-}^{(1)}$ is not ${\cal J}_{-}^{(1)}$-invariant
.

\end{proof}

\begin{cor} {\rm (\cite{BN}})~ The almost $CR$-structures $({\cal D},{\cal J}_{-}^{(1)})$, $({\cal D},{\cal J}_{\pm}^{(2)})$ are not integrable.
\end{cor}

\begin{proof}
According to Corollary~\ref{Levi-pm-n}, the almost $CR$-structures  $({\cal D},{\cal J}_{-}^{(1)})$ and
$({\cal D},{\cal J}_{+}^{(2)})$ are not integrable. Denote by $N_{-}^{(2)}$ be the Nijenhuis tensor
of the $CR$-structure $({\cal D},{\cal J}_{-}^{(2)})$. Then, using Lemma~\ref{hor-ver}, one easily computes that
$N_{-}^{(2)}((E_2)_{\sigma}^h,(\kappa_1)_p)=-4(E_5)_{\sigma}^h\neq 0$, so $({\cal D},{\cal J}_{-}^{(2)})$
is not integrable.
\end{proof}

\smallskip

\begin{prop} \label{J-Levi} If $\nabla$ is the characteristic connection,  the Levi form
$\omega^{(1)}_{+}$ is ${\cal J}^{(1)}_{+}$-invariant if and only if $\ast T\in\Lambda^2_3TM$.
\end{prop}

\begin{proof}
By Corollary~\ref{Levi-pm-n}, $\omega^{(1)}_{+}$ is ${\cal J}^{(1)}_{+}$-invariant if and only if
$$
g(T(J^{\sigma}_{+}X,J^{\sigma}_{+}Y),\xi_{\sigma})=g(T(X,Y),\xi_{\sigma})
$$
for $\sigma\in{\Bbb T}$, $X,Y\in T_{\pi(\sigma)}M$, $ X,Y\perp\xi_{\sigma}$.
\end{proof}
This, in fact, is identity (\ref{T}). As we have seen in the proof of Theorem~\ref{normal} it is equivalent to
$\ast T\in \Lambda^2_3 T_pM$.

\subsection{The Nijenhuis tensor of the almost $CR$-structure $({\cal D},{\cal J^{(1)}_{+}})$ on the twistor space}$\\$

Denote the Nijenhuis tensor of the $CR$-structure $({\cal D},{\cal J}^{(1)}_{+})$ by $N^{\it CR}$. Then we have the following.

\begin{lemma}\label{N-tens} Let $A,B\in {\cal D}_{\sigma}$ be
horizontal vectors and $V,W\in{\cal D}_{\sigma}$ vertical ones at a point $\sigma\in{\Bbb T}$.  Let
$X=\pi_{*}A, Y=\pi_{*}B$. Suppose that the Levi form $\omega^{(1)}_{+}$ is ${\cal J}$-invariant. Then:
$$
\begin{array}{l}
N^{\it CR}(A,B)= -R(X,Y)\sigma +R(J^{\sigma}_{+}X,J^{\sigma}_{+}Y)\sigma
-{\cal J}^{(1)}_{+}(R(J^{\sigma}_{+}X,Y)\sigma+R(X,J^{\sigma}_{+}Y)\sigma).\\[10pt]
 N^{\it CR}(A,V)=0, \quad N^{\it CR}(V,W)=0.
\end{array}
$$
\end{lemma}

\begin{proof} Take vector fields $X$ and $Y$ near the point
$p=\pi(\sigma)$ such that $X_p=\pi_{*}A$, $Y_p=\pi_{*}B$ and $\nabla X|_p=\nabla Y|_p=0$. Then
$-(\Phi_{+}^{(1)})^2X^{h}$ and $-(\Phi_{+}^{(1)})^2Y^{h}$ are sections of ${\cal D}$ with values
$A$ and $B$ at the point $\sigma$, so $N^{\it CR}(A,B)=N^{\it
CR}_{\sigma}((\Phi_{+}^{(1)})^2X^{h},(\Phi_{+}^{(1)})^2Y^{h})$. Then Lemma~\ref{phi-hor-brac} and
Corollary~\ref{Levi-pm-n} imply the first identity of the lemma.

  Take a section $S$ of $\Lambda^2_3TM$ near the point $p$ such that $S_p=V$ and $DS|_p=0$.
The vertical vector field $\widetilde S$ determined by $S$  takes value $V$ at the point $\sigma$.
Then  Lemma ~\ref{hor-ver} and identity (\ref{comm}) give
$$
N^{\it CR}(A,V)=-N^{\it CR}((\Phi_{+}^{(1)})^2X^h,\widetilde S)_{\sigma}=0.
$$
The restriction of ${\cal J}^{(1)}_{+}$ to any vertical space is
the complex structure of the
corresponding fibre of ${\Bbb T}$, hence $N^{\it CR}(V,W)=0$.
\end{proof}

\begin{prop} \label{Nij}
The Nijenhuis tensor $N^{\it CR}$ vanishes if and only if the $\odot^2_9$-component of ${\cal R}$
vanishes.
\end{prop}

\begin{proof}
By Lemma~\ref{N-tens} and Proposition~\ref{NXY}, the condition $N^{\it CR}=0$ is the same as ${\cal V}N^{(1)}_{+}(X^h_{\sigma},Y^h_{\sigma})=0$ for every
$\sigma\in{\Bbb T}$ and $X,Y\in H^{\sigma}$. We have seen in the proof of Theorem~\ref{normal} that the latter condition is equivalent to vanishing of the
$\odot^2_9$-component of ${\cal R}$.
\end{proof}

\subsection{Integrability of the almost $CR$-structure $({\cal D},{\cal J}^{(1)}_{+})$ on the twistor space}$\\$

 Propositions~\ref{J-Levi} and \ref{Nij} imply the following.
\begin{theorem} \rm{(\cite{BN})}\label{integr}
Suppose that the $SO(3)$-structure on $M$ is nearly integrable. Then the almost $CR$ structure $({\cal
D},{\cal J}^{(1)}_{+})$ on the twistor space ${\Bbb T}$ defined by means of the characteristic connection is
integrable if and only if $\ast T\in\Lambda^2_3TM$ and the $\odot^2_9$-component of the curvature
vanishes.
\end{theorem}

\smallskip

\noindent {\it Remark}. It is a result of Ianus \cite{I} (see
\cite[Theorem 6.57]{Blair}) that any normal almost contact metric
structure induces an integrable almost $CR$-structure. It is
well-known that the converse is not true.  The almost contact metric
structure $(\Phi^{(1)}_{+},\chi,h_t)$ in Example 2 above also gives
a counterexample. As we have seen as an application of
Theorem~\ref{normal}, it is not normal, although the induced $CR$
structure $({\cal D},{\cal J}^{(1)}_{+})$ is integrable by
Theorem~\ref{integr}.

\end{document}